\documentclass[11pt,reqno,oneside]{amsart}
\usepackage{parskip}
\usepackage[]{babel}
\usepackage[utf8]{inputenc} %Para utilizar tildes y caracteres latinos
\usepackage{amsmath,amsthm,amsbsy}
\usepackage{amsfonts}
\usepackage{amsrefs}
\usepackage{mathrsfs}
\usepackage{bbm}
\usepackage{amssymb}
\usepackage[colorlinks=true, menucolor=blue, linkcolor=blue, citecolor=blue]{hyperref}
\usepackage{enumitem}
\usepackage{multicol}
\usepackage{float}
\usepackage{fancyhdr}
\usepackage{layout}
\usepackage{esint}
\usepackage{array}
\usepackage{makecell}
\allowdisplaybreaks

\usepackage[colorinlistoftodos]{todonotes}
\usepackage{orcidlink} %Orcid

\usepackage{mathtools}
\usepackage[foot]{amsaddr}

\theoremstyle{plain}
\newtheorem{theorem}{Theorem}[section]                                          
\newtheorem{proposition}[theorem]{Proposition}                          
\newtheorem{lemma}[theorem]{Lemma}
\newtheorem{corollary}[theorem]{Corollary}

\theoremstyle{definition}
\newtheorem{definition}[theorem]{Definition}
\theoremstyle{remark}
\newtheorem{remark}[theorem]{Remark}
\newtheorem{example}[theorem]{Example}

\makeatletter \@addtoreset{equation}{section} \makeatother

\newcommand{\caract}{\mathbbm{1}}
\newcommand{\calF}{\mathcal{F}}
\newcommand{\calA}{\mathcal{A}}

\newcommand{\calB}{\mathcal{B}}
\newcommand{\calL}{\mathcal{L}}
\newcommand{\calP}{\mathcal{P}}

\newcommand{\calM}{\mathcal{M}}

\newcommand{\defeq}{:=}

\newcommand{\sbt}{\,\begin{picture}(-1,1)(-1,-3)\circle*{2}\end{picture}\ }

\newcommand{\N}{\mathbb{N}}     % Natural numbers
\newcommand{\R}{\mathbb{R}}     % Real numbers
\newcommand{\Prob}{\mathbb{P}}  % Probability measure
\newcommand{\Exp}{\mathbb{E}}   % Expectation 
\newcommand{\goth}[1]{\mathfrak{#1}} % Gothic letters 
\newcommand{\inner}[2]{\left( #1 \, , \, #2 \right)} % Inner product
\newcommand{\norm}[1]{\left\|#1\right\|}              % Vector norm
\newcommand{\triplet}[3]{\left( #1, #2, #3 \right) }             % General triplet e.g. a probability space
\newcommand{\ProbSpace}{\triplet{\Omega}{\mathcal{F}}{\Prob}}    % Triplet of a Probability Space
\newcommand{\abs}[1]{\left| #1 \right|}                          % Absolute value  

\newcommand{\quadraVari}[1]{\left\langle  #1  \right\rangle } % quadratic variation symbol

\newcommand{\fle}{\rightarrow}

\newcommand{\operQuadraVari}[1]{\left\langle \!\left\langle #1  \right\rangle \!\right\rangle} % operator quadratic variation symbol

\newcommand\restr[2]{{% we make the whole thing an ordinary symbol
  \left.\kern-\nulldelimiterspace % automatically resize the bar with \right
  #1 % the function
  \vphantom{\big|} % pretend it's a little taller at normal size
  \right|_{#2} % this is the delimiter
  }}

\title{The martingale representation theorem for cylindrical martingale valued measures}

\author{S. Cambronero\,\orcidlink{0000-0001-6758-4942}$^1$}
\author{D. Campos\,\orcidlink{0000-0002-3608-1151}$^2$}
\author{C. A. Fonseca-Mora\,\orcidlink{0000-0002-9280-8212}$^3$}
\author{D. Mena\,\orcidlink{0000-0002-9443-391X}$^4$}
\address{Centro de Investigaci\'{o}n en Matem\'{a}tica Pura y Aplicada \\ Escuela de Matem\'{a}tica, Universidad de Costa Rica}
\email{$^1$ santiago.cambronero@ucr.ac.cr} 
\email{$^2$ josedavid.campos@ucr.ac.cr}
\email{$^3$ christianandres.fonseca@ucr.ac.cr}
\email{$^4$ dario.menaarias@ucr.ac.cr}

\begin{document}
\emergencystretch 3em


\begin{abstract}
We prove a martingale representation theorem for cylindrical martingale-valued measures defined on a separable Banach space. The main tool for establishing the theorem, is a new theory of non-radonifying stochastic integration in reflexive Banach spaces. A second one is the study and characterization of cylindrical white noise measure processes. As consequences of our representation theorem, we prove analogous versions for Hilbert space-valued measures and for cylindrical square integrable martingales. Finally, we apply the results to characterize the solutions to the weak martingale problem for SDEs driven by cylindrical white noise measures.

\end{abstract}

\maketitle

%\tableofcontents
%We generalize the construction of \cite{MetivierPellaumail}, Section $16$, to the case of a cylindrical orthogonal martingale valued measure as integrator. \medskip

\section{Introduction}

The martingale representation theorem is one of the central results of modern stochastic analysis, providing a structural description of martingales in terms of stochastic integrals with respect to a fundamental driving noise. In the classical finite–dimensional setting, as well as in the infinite dimensional context, such representations play a crucial role in areas such as stochastic differential equations, filtering theory, and mathematical finance \cites{CohenElliott:2015, DaPratoZabczyk, GawaMand:2010, LeGallBrownian, MetivierPellaumail}.  In this article, our main objective is to prove a martingale representation theorem for a cylindrical orthogonal martingale-valued measure with (weakly) continuous paths, defined on a separable and reflexive Banach space. 

Cylindrical orthogonal martingale-valued measures where studied in \cite{CCFM:SPDE} with some addenda in \cite{CCFM:Ito}.  Roughly speaking, for a ring $\mathcal{A}$ of Borel subsets of a topological space $U$, a cylindrical orthogonal martingale-valued measure is a family $(M(t,A):t \geq 0, A \in \mathcal{A})$ such that  for each $A \in \mathcal{A}$, the process $(M(t,A):t \geq 0)$ is a cylindrical square integrable martingale on a Banach space $X$ and for each $t \geq 0$, $M(t, \cdot)$ is an $L^{2}$-valued finitely additive measure on $\mathcal{A}$. This concept generalizes
both the Hilbert-space valued martingale-valued measures and the cylindrical square integrable martingales (see Section 6 in \cite{CCFM:Ito}). In particular, these include some of the popular models of stochastic integrators in the literature, as for example the White noise (in the sense of martingale measures introduced by Walsh in \cite{Walsh:1986}), square integrable Hilbert space-valued L\'evy processes, stochastic integrals defined with respect to Poisson random measures, and cylindrical square integrable L\'evy processes (see \cite{CCFM:SPDE}). 

Within the context of Hilbert spaces, stochastic integrals with respect to cylindrical orthogonal martingale-valued measures were constructed in \cite{CCFM:SPDE}. The corresponding class of integrands consists of Hilbert-Schmidt operators which are locally square integrable with respect to the (predictable) quadratic variation $\operatorname{M}$ and the operator  quadratic variation $Q_{M}$ of $M$ (both $\operatorname{M}$ and $Q_{M}$ were constructed in  \cite{CCFM:SPDE}).  The constructed stochastic integral is a  locally square integrable martingale with values in a Hilbert space. 

In Section \ref{sectPreliminaries}, we review the main definitions and properties of cylindrical martingale-valued measure, as well as the theory of integration with respect to those objects, as developed in \cite{CCFM:SPDE}.   In Section \ref{SectionNonRadonifyingIntegral}, we extend the construction  in \cite{CCFM:SPDE} to define a stochastic integral which is itself a cylindrical orthogonal martingale-valued measure, within the context of reflexive Banach spaces. Since the idea is to define a cylindrical object, we refer to this integral as the non-radonifying stochastic integral.

The main steps in the construction of the integral are carried out in Section \ref{sectStochIntegAsACylMVM}. To define the corresponding class of integrands, we extend the arguments used in \cite{MetivierPellaumail}*{Section 16} to our more general context. Our main result in this section is Theorem \ref{theoExistenceStochIntegral}, which shows that the non-radonifying integral is a cylindrical martingale-valued measure. 
Some elementary properties of the integral are derived in Section \ref{subsecPropertiesIntegral}. The existence of the quadratic variation of the non-radonifying integral is addressed in Section \ref{subSectQuadraVariOfIntegral} where a explicit formula is derived (see \eqref{eqQuadraticVariationN}) as well as the corresponding operator quadratic variation (Theorem \ref{theoDensities}). An example is given in Section \ref{subsectWhiteNOiseBanach} where we introduce a white noise measure on a Banach space. Since the non-radonifying integral is itself a cylindrical martingale-valued measure, it is normal to consider stochastic integrals with respect to it; this leads naturally to the study of associativity properties of the stochastic integral. This problem is studied in Section \ref{subsectAssociaStochIntegral} and it turns out to be of great importance to establish the martingale representation theorem. As a final step in our study, within the context of separable Hilbert spaces we show in Section \ref{subSectCylinVsRadonifiedIntegral} that the stochastic integral constructed in \cite{CCFM:SPDE} is indeed a radonified version of the non-radonifying integral when considering Hilbert-Schmidt integrands. Aside from the theoretical interest on this result, it has a latter application to derive a Hilbert space-valued martingale representation theorem. 

Section \ref{sectionMartinRepreTheorem} contains the main results of this paper. We consider a cylindrical martingale-valued measure 
%$M=(M(t, A): t \geq 0, A \in \calA)$ 
with (weakly) continuous paths and show that, under some standard assumptions on its covariance structure, it can be represented as a non-radonifying stochastic integral with respect to some white noise measure. 

In Section \ref{subSectCyliWhiteNoise} we begin by introducing the concept of a cylindrical $Q$-white noise measure 
$W=(W(t, A): t \geq 0, A \in \calA)$ defined on a separable Hilbert space $H$. This object is defined through a random series representation of real-valued white noise measures (see \eqref{eqDefWN}), whose covariance operator $Q$ has a (weak) representation as a series with respect to some orthonormal basis. We explore basic properties of the  cylindrical $Q$-white noise measure as for example its (predictable) quadratic variation and its connection with the Hilbert-space valued white noise measure. In Section \ref{subSectLevyCharact} we prove a L\'evy characterization theorem for the cylindrical $Q$-white noise measure.  

In Section \ref{subsectMartingaleRepresentation} we prove our martingale representation (Theorem \ref{martingalerepresentationtheoremContinuousCase}) which can be formulated as follows: Let $H$ be a separable Hilbert space and $X$ a separable, reflexive Banach space. Assume that $M$ defined on $X^{*}$ has a covariance structure of the form:
$$ \quadraVari{M(t,A)x^*,M(t,A)y^*}_{t}=\int_0^t\!\! \int_A\; \quadraVari{\Phi (r,u) Q\Phi^*(r,u)x^{*}, y^{*}} \lambda(du) dr,  $$
where $x^{*}, y^{*} \in X^{*}$, $Q$ is a bounded operator on $H$ with some series representation (see \eqref{representationCovariaQDefiCylinWhineNoiseMeasure}), $\Phi: \Omega \times[0,T] \times U \rightarrow \mathcal{L}(H,X)$ satisfies some predictability assumption and $\Phi(\omega,t,u)Q\Phi^*(\omega,t,u)$ is a compact operator. Then, there exists a 
cylindrical $Q$-white noise measure $W=(W(t, A): t \geq 0, A \in \calA)$ on $H$ such that 
on some extended probability space we have
 $$
M(t,A)= \int_0^t\!\! \int_A \Phi(\omega,r,u)W(dr,du),
$$ 
where the equality should be understood in the sense of cylindrical stochastic processes. 

As an application of our main result, in Section \ref{sectSpecialMartingaleRepre} we prove a martingale representation theorem for Hilbert space-valued orthogonal martingale-valued measures (see Theorem \ref{theoremClassicalMartingaleRepresentation}) and a martingale representation theorem for cylindrical continuous square integrable martingales defined on a separable reflexive Banach space (see Theorem \ref{theoreprescylindricalmartingale}). To the extent of our knowledge, the first of these results is completely new in the literature. The second one, is comparable to the results obtained for example in \cite{Ondrejat:2005} and \cite{VeraarYaroslavtsev:2016}.  We end this section by deriving a representation theorem for classical square integrable martingales taking values on a reflexive, separable Banach space (Theorem \ref{theoRepClassMartBan}), this extends in some sense previous results in the context of Hilbert spaces, as for example the one in \cite{DaPratoZabczyk}.

Finally, in Section \ref{sectWeakMartingaleProblem} we study weak (distributional) solutions to the stochastic differential equation 
\begin{equation}\label{eqSDEIntro}
   dX_t = a(t,X_t)\,dt + \int_U b(t,u,X_t)\, W(dt,du),    
\end{equation}
where $W$ is a cylindrical $Q$-white noise and $a$ and $b$ are suitable coefficients.  
We define the associated martingale problem, and apply our martingale representation theorem to show (see Theorem \ref{theoWeakSolutionEquival}) that a weak solution to \eqref{eqSDEIntro} exists if and only if there is a solution to the corresponding martingale problem. We hope that this result can be applied to show existence and uniqueness of a weak solution to \eqref{eqSDEIntro}.  These results will be considered to appear elsewhere. 

\section{Preliminaries}\label{sectPreliminaries}

\subsection{Linear spaces and cylindrical processes}

We assume that $(\Omega, \mathcal{F}, \Prob)$ is a complete probability space equipped with a filtration $(\mathcal{F}_{t})_{t \in \R_{+}}$ that satisfies the \emph{usual conditions}, i.e. it is right continuous and $\mathcal{F}_{0}$ contains all $\Prob$-null sets. The predictable $\sigma$-algebra on $\Omega \times [0,\infty)$ is denoted by $\mathcal{P}$ and for any $T>0$ we denote by $\mathcal{P}_{T}$ the restriction of $\mathcal{P}$ to $\Omega \times [0,T]$ (For further details, see chapters 6 and 7 in \cite{CohenElliott:2015}).

Given a Banach space $X$ with strong dual $X^{*}$, the canonical dual pairing will be denoted by $\langle x, x^* \rangle_{X,X^*}$ for $x \in X$ and $x^{*} \in X^*$. A \emph{cylindrical random variable} on $X^{*}$ is a linear and continuous operator $Z:X^{*} \rightarrow L^{0} (\Omega, \mathcal{F}, \Prob)$, where $L^{0} (\Omega, \mathcal{F}, \Prob)$ is the space of (equivalence classes of) real-valued random variables, equipped with the topology of convergence in probability. A family of cylindrical random variables $Z=(Z_{t}: t \in [0,T])$ on $X^*$ is called a \emph{cylindrical stochastic process} on $X^*$. A cylindrical stochastic process  $M=(M_{t}: t \in [0,T])$ is called a \emph{cylindrical mean-zero square integrable martingale} on $X$ if for every $x^* \in X^*$ we have $M(x^*) \in \mathcal{M}^{2}$, the linear space of all real-valued, c\`{a}dl\`{a}g, mean-zero, square integrable martingales on the time interval $[0,T]$. This is a Banach space with the norm 
$$
\norm{m}_{\mathcal{M}^{2}}= \sup_{t \in [0,T]} \left( \Exp \left[  \abs{m(t)}^{2} \right] \right)^{1/2}.
$$
For a Hilbert-space valued square-integrable martingale $m=(m_{t}: t \geq 0)$, we denote by $\quadraVari{m}=(\quadraVari{m}_{t}: t \geq 0)$ its predictable quadratic-variation (i.e. the predictable compensator of $(\norm{m_{t}}^{2}: t \geq 0)$) and we denote by $[m]=([m]_{t}: t \geq 0)$ the (optional) quadratic variation of $m$. For details on the definition and properties of  $\quadraVari{m}$ and $[m]$ see e.g.  \cite{Metivier}. 

For any two Banach spaces $X$ and $Y$, the Banach space of bounded linear operators from $X$ into $Y$ will be denoted by  $\mathcal{L}(X,Y)$. We will denote the space of real-valued bounded bilinear forms on $X \times Y$ by $\goth{Bil}(X,Y)$. Trace class operators on a Hilbert space $H$ will be denoted by $\mathcal{L}_{1}(H)$.

\subsection{Cylindrical martingale-valued measures and quadratic variation}\label{sectCMVM}

Let $X$ be a Banach space with separable dual $X^{*}$. Throughout this work we assume that $M$ is a \emph{cylindrical orthogonal martingale-valued measure} on $X^*$ (see \cite{CCFM:SPDE}). To be precise, we fix a Hausdorff topological space $U$, which is Lusin in the sense that it is
homeomorphic to a Borel subset of the line. We consider a ring $\mathcal{A}$ of Borel subsets of $U$. A \emph{cylindrical martingale-valued measure} $M$ is defined as a collection  $(M(t,A): t \geq 0, A \in \mathcal{A})$ of cylindrical random variables on $X^*$ such that
\begin{enumerate} 
\item For each $A \in \mathcal{A}$, $M(0,A)(x^*)= 0$ $\Prob$-a.e. for all $x^* \in X^*$. \label{timezerocomvmdef}
\item For each $A \in \mathcal{A}$, $M(A) = (M(t,A): t \geq 0)$, is a cylindrical mean-zero square integrable martingale, and for each $t > 0$ and $A \in \mathcal{A}$, the map 
$$ 
M(t,A): X^* \rightarrow L^{0} \ProbSpace
$$ 
is continuous. \label{cylindricalMartingale}
\item If $t>0$ and $x^* \in X^*$, $M(t,\cdot)(x^*): \mathcal{A} \rightarrow L^{2} \ProbSpace$ is a $\sigma$-finite $L^{2}$-valued measure. As part of the hypothesis, $\mathcal{A}$ contains each $\mathcal{B}(U_n)$, where $(U_n)$ is the sequence of sets increasing to $U$, given by the $\sigma$-finiteness of $M(t,\cdot)h$. This sequence is independent of $t$. (see Definition 3.1 in \cite{CCFM:SPDE}). \label{fixedtimeismeasure}
\end{enumerate}

We further say that $M$ is \emph{orthogonal} if: 
\begin{enumerate}\setcounter{enumi}{3}
\item Given $t>0$ and $x^* \in H$, $\quadraVari{M(A)(x^*),M(B)(x^*)}_{t}=0$ whenever $A,B\in \mathcal{A}$ are disjoint. \label{orthogonality}
\end{enumerate}

Associated to $M$ there exists a collection random predictable $\sigma$-finite measures $(\nu_{x^*}: x^* \in X^*)$, called the \emph{intensity measures} of $M$, with 
the property that for every $t \geq 0$  and $A \in \mathcal{A}$, we have $ \nu_{x^{*}} (\omega) ([0, t] \times A) = \quadraVari{M(A)(x^{*})}_{t}(\omega)$ $\Prob$-a.e. 

We will assume that $M$ has a unique (predictable) quadratic variation (see Theorem 5.10 in \cite{CCFM:SPDE} on sufficient conditions for its existence), that is, a random measure $\operQuadraVari{M}: \Omega \rightarrow \mathcal{M}_{+}(\R_{+} \times U, \mathcal{B}(\R_{+}) \otimes \mathcal{B}(U) )$ that satisfies: 
\begin{enumerate}    
    \item Given $t \geq 0$ and $A \in \mathcal{A}$, for $\Prob$-a.e. $\omega \in \Omega$ we have $\operQuadraVari{M}(\omega)([0,t] \times A) < \infty$.
    \item \label{propertyUpperBoundNuX} $\operQuadraVari{M}$ is a minimal element  for the collection of all random measures $\zeta: \Omega \rightarrow \mathcal{M}_{+}(\R_{+} \times U, \mathcal{B}(\R_{+}) \otimes \mathcal{B}(U) )$ with the property:   $\forall\, x^{*} \in X^{*}$ with $\| x^{*} \| = 1$,  $\nu_{x^{*}} \leq \zeta$. 
\end{enumerate} 
We denote property \ref{propertyUpperBoundNuX} as $\displaystyle \operQuadraVari{M}=\sup_{\norm{x^{*}}=1} \nu_{x^{*}}$. 

We will further assume that the quadratic variation of $M$ satisfies: for $\Prob$-a.e. $\omega \in \Omega$,
\begin{equation}\label{eqBoundedrangequadraticvariation}
    \sup_{A \in \calA} \operQuadraVari{M}(\omega)([0,T]\times A) <\infty.
\end{equation}
Then by  Theorem 5.21 in \cite{CCFM:SPDE}, given $T>0$,  there exists a  process $Q_{M}: \Omega \times [0,T]  \times U \rightarrow \mathcal{L}(X^{*},X^{**})$ such that for all $x^{*}_{1},x^{*}_{2} \in X^*$, $0 \leq t\leq T$ and $A \in \mathcal{A}$, $\Prob$-a.e. $\omega \in \Omega$,
\begin{equation}\label{existenceofQ}
\quadraVari{M(A)(x^{*}_{1}),M(A)(x^{*}_{2})}_{t}(\omega) = \int_{[0,t]\times A} \langle Q_{M}(\omega,r,u)x^{*}_{1},x^{*}_{2} \rangle_{X^{**},X^{*}} \, \operQuadraVari{M} (\omega)(dr,du)  
\end{equation}
for all $x^{*}_{1}, x^{*}_{2} \in X^{*}$, $C \in \mathcal{B}([0,T]) \otimes \mathcal{B}(U)$. Moreover, the following properties hold:
\begin{enumerate}
    \item \label{propPredictQM} For every $x^{*}_{1}, x^{*}_{2} \in X^{*}$, the mapping $(\omega,r,u) \mapsto \langle Q_{M}(\omega,r,u)x^{*}_{1},x^{*}_{2}\rangle_{X^{**},X^{*}}$ is predictable, that is, $\mathcal{P}_T \otimes \mathcal{B}(U)$-measurable. \label{Qpredictable}
    \item \label{propPositSymmetQM} For $\Prob$-a.e. $\omega \in \Omega$, $Q_{M}(\omega,\cdot,\cdot)$ is positive and symmetric $\operQuadraVari{M}$-a.e. \label{Qpositiveandsymmetric}
    \item \label{propNormOneQM} For $\Prob$-a.e. $\omega \in \Omega$, $\norm{Q_{M}(\omega,\cdot,\cdot)}_{\mathcal{L}(X^{*},X^{**})}=1$,  $\operQuadraVari{M}$-a.e. \label{normoneoftheoperatorQ}
\end{enumerate}

The \emph{Dol\'eans measure} $\mu_M$ associated to $M$ on  $(\Omega\times [0,T]\times U,\mathcal{P}_{T}\otimes \mathcal{B}(U)) $ is defined as
$$
\mu_M(C) = \mathbb{E} \int_{[0,T]\times U} \caract_{C}\, d\operQuadraVari{M}, \quad C \in \mathcal{P}_{T}\otimes \mathcal{B}(U). 
$$
If $g: \Omega \times [0,T] \times U \rightarrow [0,\infty)$ is predictable, one can show that (see Lemma 6.2 in \cite{CCFM:SPDE}) 
\begin{equation}\label{eqEqualIntegralDoleans}
\int_{\Omega \times [0,T] \times U} g(\omega, s,u) \,  d \mu_{M}(\omega, s,u)
= \mathbb{E} \left[ \int_{[0,T]\times U} g(\cdot, s,u) \,  \operQuadraVari{M}(ds, du) \right].    
\end{equation}

\section{The non-radonifying stochastic integral}\label{SectionNonRadonifyingIntegral}

\subsection{Stochastic integral as a cylindrical martingale-valued measure}\label{sectStochIntegAsACylMVM}

From now on, we assume that $X$ is a reflexive and separable Banach space and that $Y$ any Banach space with $Y^*$ separable. Moreover, we assume  $M=(M(t,A): t \geq 0, A \in \mathcal{A})$ a cylindrical orthogonal martingale-valued measure on $X$ satisfying the assumptions listed in Section \ref{sectCMVM}. 
We denote, as usual
$$
M((s,t],A)(x^*) = M(t,A)(x^*) - M(s,A)(x^*),\quad x^*\in X^*.
$$

The main goal in this section is to introduce a sufficiently large class of operator-valued processes $\Phi=(\Phi(\omega,t,u))$ for which one can construct a new cylindrical orthogonal martingale-valued measure via stochastic integration of $\Phi$ with respect to $M$. We start by considering stochastic integration with respect to simple families. 

We consider a family $\Phi$ of operators $\Phi(\omega,t,u) \in \calL(X,Y)$, indexed by $(\omega,t,u)\in \Omega\times [0,T]\times U$. This family is called \emph{simple} if it can be written as
\begin{equation}
\label{eqDefSimpleFamily}
   \Phi = \sum_{i=1}^N \caract_{F_i \times (s_i,t_i]\times A_i} S_i 
\end{equation}

where $S_i\in \calL(X,Y)$, $F_i\in \calF_{s_i}$ and $A_i\in \calA$ for each $i$. For $i\neq j$ one of the following alternatives holds:
\begin{itemize}
    \item $(s_i,t_i]\cap (s_j,t_j] = \emptyset$
    \item $(s_i,t_i] = (s_j,t_j]$ and $A_i\cap A_j =\emptyset$.
\end{itemize}
This is the way we avoid having a double sum in the definition of a simple family.

%We assume $M$ has a quadratic variation and its family of intensity measures satisfies the sequential boundedness property (as defined on \cite{CCFM:SPDE}). In particular, the quadratic variation $\operQuadraVari{M}$  is unique (Theorem $5.10$ in \cite{CCFM:SPDE}). We also assume that, for $\Prob$-a.e. $\omega\in\Omega$
% \begin{equation}\label{eqBoundedrangequadraticvariation}
%     \sup_{A \in \calA} \operQuadraVari{M}(\omega)([0,T]\times A) <\infty.
% \end{equation}

We define the integral of $\Phi$ with respect to $M$ as the cylindrical process $I^{\Phi}$ on $Y^*$ given by
\begin{equation}\label{eqDefiIntegSimpleProce}
I^{\Phi}_t(A)y^* := \sum_{i=1}^N \caract_{F_i} M((t\land s_i,t\land t_i],A\cap A_i) (S_i^*y^*),\quad A\in \calA,\, y^*\in Y^*.    
\end{equation}
For $A\in \calA$, $y^*\in Y^*$, the norm of the real valued square integrable martingale $I^{\Phi}(A)y^*$ is
$$
\norm{I^{\Phi}(A)y^*}_{\calM^2}^2 = \Exp | I^{\Phi}_{T}(A)y^* |^2 = \sum_{i=1}^N \Exp \left[ \caract_{F_i} M^2 \left( (s_i,t_i]\times A\cap A_i \right) (S_i^*y^*) \right],
$$
where all cross terms vanish because of the martingale property and orthogonality of $M$. 
Since $F_i\in \calF_{s_i}$ and $t\mapsto M^2_{t}(A\cap A_i)(S_i^*y^*) - \quadraVari{M(A\cap A_i)(S_i^*y^*)}_{t}$ is a martingale, we have
$$
\norm{I^{\Phi}(A)y^*}_{\calM^2}^2 = \sum_{i=1}^N \Exp \left[ \caract_{F_i} \quadraVari{M(A\cap A_i)(S_i^*y^*)}^{t_i}_{s_i} \right]
$$
and using \eqref{existenceofQ} we have 
$$
\norm{I^{\Phi}(A)y^*}_{\calM^2}^2 = \Exp \int_{[0,T]\times A} \quadraVari{Q_M \Phi^*y^*,\Phi^* y^*}_{X,X^*} d\operQuadraVari{M},
$$
where the process $Q_{M}: \Omega \times [0,T]  \times U \rightarrow \mathcal{L}(X^*,X)$ is defined via \eqref{existenceofQ} and satisfies \ref{propPredictQM}, \ref{propPositSymmetQM} and \ref{propNormOneQM} (here and from now on we identify $X^{**}$ with $X$).

%Theorem $5.21$ in \cite{CCFM:SPDE}
% have the following properties:
% \begin{enumerate}
%     \item For every $x_1^*, x_2^* \in X^*$,  $(\omega,r,u) \mapsto \quadraVari{Q_{M}(\omega,r,u)x_1^*,x_2^*}_{X,X^*}$ is predictable.
%     \item For $\Prob$-a.e. $\omega \in \Omega$, $Q_{M}(\omega,\cdot,\cdot)$ is positive and symmetric $\operQuadraVari{M}$-a.e. \label{Qpositiveandsymmetric}
%     \item For $\Prob$-a.e. $\omega \in \Omega$, $\norm{Q_{M}(\omega,\cdot,\cdot)}_{\mathcal{L}(X^*,X)}=1$ $\operQuadraVari{M}$-a.e. \label{normoneoftheoperatorQ}
% \end{enumerate}

Notice that by \eqref{eqEqualIntegralDoleans} we also have 
\begin{eqnarray}
\nonumber
    \norm{I^{\Phi}(A)y^*}_{\calM^2}^2 & = & 
    \Exp \int_{[0,T]\times A} \quadraVari{\Phi Q_M \Phi^*y^*, y^*}_{Y,Y^*} d\operQuadraVari{M}\\
    & = & 
\label{eqNormOfIntegralSimple}
\int_{\Omega\times [0,T]\times A} \quadraVari{\Phi Q_M \Phi^*y^*,y^*}_{Y,Y^*} d\mu_M
\end{eqnarray}
where $\mu_M$ is Dol\'eans measure associated to $M$.

\begin{remark}
In case $X=H$ and $Y=G$ are Hilbert spaces we get, for $g\in G$
$$
\norm{I^{\Phi}(A)g}_{\calM^2}^2 = \Exp \int_{[0,T]\times A} \norm{Q_M^{1/2} \Phi^*g}_{H}^2 d\operQuadraVari{M} =
\int_{\Omega\times [0,T]\times A} \norm{Q_M^{1/2} \Phi^*g}_{H}^2 d\mu_M.
$$    
\end{remark}

% \section{Construction of the integral}

As the second step in our construction, we introduce the main class of integrands. 

\begin{definition}
We consider the space $L = L(M,T,X,Y)$ whose elements are families of operators $\Phi(\omega,t,u)$ that satisfy:
\begin{enumerate}
    \item For each $(\omega,t,u)$, $\Phi(\omega,t,u) \in \mathcal{L}(X,Y)$.
    \item The process $\quadraVari{Q_M \Phi^* y_1^*,\Phi^* y_2^*}_{X,X^*}$ is predictable for each $(y_1^*,y_2^*) \in Y^*\times Y^*$.
    % The bilinear form $(y_1^*,y_2^*) \mapsto \quadraVari{Q_M \Phi^* y_1^*,\Phi^* y_2^*}_{X,X^*}$ is continuous on $Y^*\times Y^*$ and the process $\quadraVari{Q_M \Phi^* y_1^*,\Phi^* y_2^*}_{X,X^*}$ is predictable for each $(y_1^*,y_2^*) \in Y^*\times Y^*$.
    \item \label{defiNormMPhiA} For each $A\in \mathcal{A}$ the expression
    \begin{equation}
    \label{eqNormMPhi}
    \rho_{M,A}(\Phi) := \sup_{\norm{y^*}_{Y^*}\leq 1} \left[ \int_{\Omega\times [0,T]\times A} \quadraVari{Q_M \Phi^* y^*,\Phi^* y^*}_{X,X^*} d\mu_M \right]^{1/2}
    \end{equation}
    represents a (finite) real number, that is $\rho_{M,A}(\Phi)<\infty$.
\end{enumerate}
\end{definition}

\begin{remark}
Consider 
$$
\rho_M(\Phi) := \sup_{A\in\mathcal{A}} \rho_{M,A}(\Phi) = \sup_{\norm{y^*}_{Y^*}\leq 1} \left[ \int_{\Omega\times [0,T]\times U} \quadraVari{Q_M \Phi^* y^*,\Phi^* y^*}_{X,X^*} d\mu_M \right]^{1/2}.
$$
The hypothesis $\rho_M(\Phi)<\infty$ is more restrictive than (iii).
\end{remark}

We denote by $\mathcal{S} = \mathcal{S}(M,X,Y)$ the subset of $L$ whose elements are simple families. It is not difficult to check that each $\rho_{M,A}$ is a  seminorm on $L$. 
Hypothesis \ref{defiNormMPhiA} allows us to define a pseudo-metric
$$
d_M(\Phi,\Psi) = \sum_{n=1}^\infty 2^{-n}\frac{\rho_{M,U_n}(\Phi-\Psi)}{1+\rho_{M,U_n}(\Phi-\Psi)}.
$$
In what follows, we implicitly identify two families $\Phi$ and $\Psi$ whenever $d_M(\Phi-\Psi)=0$. Since $X$ is reflexive, this is equivalent to
$$
\forall y^*\in Y^* \quad (\Phi-\Psi)Q_M(\Phi-\Psi)^*y^* = 0 \quad \mu_M\text{- a.e.}
$$
We denote $\Lambda = \Lambda(M,X,Y)$ the closure of $\mathcal{S}$ under $d_M$. Notice that $\Phi\in \Lambda$ if and only if there is a sequence $(\Phi)_n$ in $\mathcal{S}$ such that $\rho_{M,A}(\Phi_n-\Phi)\fle 0$ for each $A\in\calA$.

From \eqref{eqNormOfIntegralSimple}, for each $A\in \calA$, $y^*\in Y^*$ and $\Phi \in \Lambda(M,T,X,Y)$ we can define
$$
I^{\Phi}(A)y^* := \calM^2\text{-}\lim_{n\fle\infty} I^{\Phi_n}(A)y^*,
$$
where $(\Phi_n)$ is any sequence in $\mathcal{S}$ such that $d_M(\Phi_n-\Phi)\fle 0$. 
For each $A\in \calA$, $I^{\sbt}(A)$ maps $(\Lambda,\rho_{M,A})$ isometrically into a subspace of $\calL(Y^*,\calM_T^2)$. In fact, from \eqref{eqNormOfIntegralSimple} we have
\begin{equation}
\label{eqIsometria}
    \norm{I^{\Phi}(A)}_{\calL(Y^*,\calM_T^2)} = \sup_{\norm{y^*}_{Y^*}\leq 1} \norm{I^{\Phi}(A)y^*}_{\calM^2} = \rho_{M,A}(\Phi).
\end{equation}

%\section{The integral as a cylindrical orthogonal martingale-valued measure}

Clearly we have 
$I_0^{\Phi}(A)y^* = 0$ for $A\in \calA$, $y^*\in Y^*$ and $\Phi \in \mathcal{S}$, therefore for any $\Phi \in \Lambda$. It is also clear that, for $\Phi\in \mathcal{S}$ and $A\in \calA$, $I^{\Phi}(A)$ is a cylindrical mean-zero square integrable martingale. Besides
$$
I^{\Phi}_t(A):Y^*\fle L^0(\Omega,\calF,\Prob)
$$
is continuous. This property translates via limits to the integral of any $\Phi \in \Lambda$.\medskip

Denote by $\mathbb{M}^2(T,X)$ the space of cylindrical orthogonal martingale-valued measures defined on $\calA$,  taking values on $X$ and satisfying the assumptions listed in Section \ref{sectCMVM}. 

\begin{definition}
We denote $I^{\Phi}=(I^{\Phi}_{t}(A): t \geq 0, A \in \mathcal{A})$ by $\int \Phi \, dM$ and refer to it as the \emph{stochastic integral} of $\Phi$ with respect to $M$. 
\end{definition}

\begin{theorem}\label{theoExistenceStochIntegral}
For $M\in \mathbb{M}^2(T,X)$ and $\Phi\in \Lambda(M,X,Y)$, $\int \Phi dM$ belongs to $\mathbb{M}^2(T,Y)$.    
\end{theorem}

\begin{proof}
Given $t>0$, $y^*\in Y^*$ and $\Phi\in\Lambda$, $I^{\Phi}_t(\sbt)y^*:\calA \fle L^2(\Omega,\calF,\Prob)$ is a $\sigma$-finite $L^2$-valued measure. In fact, Definition 3.1. in \cite{CCFM:SPDE} can be directly verified when $\Phi$ is a simple family and then, through the limit it is extended to the general case $\Phi\in \Lambda$. For instance, to verify that $I^{\Phi}_t(A)y^*$ is the $L^2$-limit of $I^{\Phi}_t(A\cap U_m)y^*$ we first write $\Phi$ as the $d_M$-limit of a sequence of simple families $\Phi_n$. The result follows from the inequalities
\begin{eqnarray*}
\norm{I^{\Phi}_t(A)y^* - I^{\Phi}_t(A\cap U_m)y^*}_2 & \leq & \norm{I^{\Phi}_t(A)y^* - I^{\Phi_n}_t(A)y^*}_2 + \norm{I^{\Phi_n}_t(A)y^* - I^{\Phi_n}_t(A\cap U_m)y^*}_2 \\
&  + & \norm{I^{\Phi_n}_t(A\cap U_m)y^* - I^{\Phi}_t(A\cap U_m)y^*} \\
& \leq & 2 \rho_{M,A}(\Phi - \Phi_n) + \norm{I^{\Phi_n}_t(A)y^* - I^{\Phi_n}_t(A\cap U_m)y^*}_2 .
\end{eqnarray*}

Finally, we must verify that $I^{\Phi}$ is orthogonal. Again, it is enough to prove it for $\Phi$ simple. In fact, for $\Phi$ given by \eqref{eqDefSimpleFamily}, $y^*\in Y^*$, $t>s$, $F\in \calF_s$ and $A\cap B = \emptyset$, it is straightforward to verify that 
$$
\Exp \left[ \caract_F I_t^{\Phi}(A)y^* \cdot I_t^{\Phi}(B)y^*   \right] = \Exp \left[ \caract_F I_s^{\Phi}(A)y^* \cdot I_s^{\Phi}(B)y^*   \right].
$$
This is obtained by orthogonality and martingale property of $M$. This shows that $I_t^{\Phi}(A)y^* \cdot I_t^{\Phi}(B)y^*$ is a martingale, so in particular $\quadraVari{I_t^{\Phi}(A)y^*,I_t^{\Phi}(B)y^*}_t = 0$. \end{proof}

Consider the space 
$$
L^{2,loc} = L^{2,loc}(M,X,Y) := \bigcap_{n\in\N} L^{2}_{\calL(X,Y)}\left( \Omega\times [0,T]\times U_n, \calP \otimes \calB(U_n),\mu_{M,n}  \right)
$$
as a subspace of $L^{0}_{\calL(X,Y)}\left( \Omega\times [0,T]\times U, \calP \otimes \calB(U),\mu_{M} \right)$,
where $\mu_{M,n}$ is the restriction of $\mu_M$ to $\calB(U_n)$. This is a Fr\'echet space with the metric
$$
d_2(\Phi,\Psi) = \sum_{n=1}^\infty 2^{-n}\frac{\norm{\Phi-\Psi}_{2,n}}{1+\norm{\Phi-\Psi}_{2,n}}.
$$
In the above we have 
$$ \norm{\Phi}^{2}_{2,n} = \int_{\Omega\times [0,T]\times U_n} \norm{\Phi}_{\calL(X,Y)}^2 d\mu_{M,n}. $$

The following result shows that any $\mu_M$-locally square integrable family belongs to $\Lambda(M,X,Y)$.

\begin{theorem}
\label{theoCondL2LocInteg}
The space $L^{2,loc} = L^{2,loc}(M,X,Y)$ is contained in $\Lambda(M,X,Y)$ and, for $\Phi\in L^{2,loc}$ we have
$$
d_{M}(\Phi) \leq d_{2}(\Phi).
$$
\end{theorem}

\begin{proof}
In fact, for $\Phi\in L^{2}_{\calL(X,Y)}\left( \Omega\times [0,T]\times U_n, \calP \otimes \calB(U_n),\mu_{M,n}  \right)$ we have
\begin{eqnarray*}
\rho_{M,U_n}(\Phi)^2 & \leq &  \int_{\Omega\times [0,T]\times U_n} \sup_{\norm{y^*}_{Y^*}\leq 1} \quadraVari{Q_M \Phi^* y^*,\Phi^* y^*}_{X,X^*} d\mu_{M,n} \\
 & \leq & \int_{\Omega\times [0,T]\times U_n} \norm{Q_M}_{\calL(X^*,X)} \norm{\Phi^*}_{\calL(Y^*,X^*)}^2 d\mu_{M,n} \\
  & = & \int_{\Omega\times [0,T]\times U_n} \norm{\Phi}_{\calL(X,Y)}^2 d\mu_{M,n} = \norm{\Phi}_{2,n}.
\end{eqnarray*}
This shows that, for $\Phi\in L^{2,loc}$,   $d_M(\Phi) \leq d_2(\Phi)$. Since $L^{2,loc}$ is contained in the closure of $\mathcal{S}$ under $d_2$, we are done.
\end{proof}

\subsection{Elementary properties of the non-radonifying integral}\label{subsecPropertiesIntegral}

We begin with the linearity property of the stochastic integral on the integrands. 

\begin{proposition}\label{propLinearityIntegral}
 Let $\Phi, \Psi\in\Lambda(M,X,Y)$ and $\lambda \in \R$, we have
$$
\int (\lambda \Phi + \Psi)\; dM = \lambda \int \Phi\;dM + \int \Psi\; dM. 
$$
More precisely, for any given   $A\in\mathcal{A}$, the identity
$$
\int_0^t\int_A (\lambda \Phi(s,u) + \Psi(s,u))\; M(ds,du) = \lambda \int_0^t\int_A \Phi(s,u)\;M(ds,du) + \int_0^t\int_A \Psi(s,u)\; M(ds,du)
$$
holds $\Prob$-a.e. simultaneously for all $t\in [0,T]$.   
\end{proposition}
\begin{proof}
This integral is clearly linear in the space of simple families and therefore in all $\Lambda(M,X,Y)$ by a classical approximation argument.    
\end{proof}

Let $Z$ be a Banach space with separable dual and consider $R \in \mathcal{L}(Y,Z)$. For each $\Phi \in \Lambda(M,X,Y)$ we define
$$
R\circ \Phi := \{ R \circ \Phi(\omega,r,u): (\omega,r,u) \in \Omega\times [0,T]\times U \}.
$$
If $R\neq 0$ and $\Phi\in \Lambda(M,X,Z)$ we have
\begin{equation}
\label{eqRhoMRoPhi}
\rho_{M,A}(R\circ \Phi) \leq \norm{R} \rho_{M,A}(\Phi)
\end{equation}
and therefore $R\circ \Phi \in \Lambda(M,X,Z).$ In fact, \eqref{eqRhoMRoPhi} follows from 
$$
\quadraVari{Q_M \Phi^* R^*z^*,\Phi^* R^*z^*}_{X,X^*} = \norm{R^*}^2 \quadraVari{Q_M \Phi^*y^*,\Phi^*y^*}_{X,X^*},
$$
where 
$$
\norm{y^*}_{Y^*} = \norm{R^*\left( \frac{z^*}{\norm{R^*}} \right)}_{Y^*} \leq \norm{z}_{Z^*} \leq 1.
$$
\begin{proposition}
\label{propMappingIntegContOpera} Let $Z$ be a Banach space with separable dual and consider $R \in \mathcal{L}(Y,Z)$. For each $\Phi \in \Lambda(M,X,Y)$ we have $R\circ \Phi \in \Lambda(M,X,Z)$ and $\Prob$-a.e. for every $t \in [0,T]$ and every $A\in \mathcal{A}$,
\begin{equation} \label{eqIntUnderContMapping}
\int^{t}_{0}\!\! \int_{A} R \circ \Phi (r,u) M (dr, du)= R^{**} \left(\int^{t}_{0}\!\! \int_{A} \Phi (r,u) M (dr, du) \right). 
\end{equation} 
\end{proposition}
\begin{proof} One can easily  check, using \eqref{eqRhoMRoPhi}, that $R \circ \Phi \in \Lambda(M,X,Z)$. It suffices to check that  \eqref{eqIntUnderContMapping} holds for elementary $\Phi$. In fact, if $\Phi = \caract_{F\times (s_0,t_0]\times B}S$, with $F\in \mathcal{F}_{s_0}$ and $S\in \mathcal{L}(X,Y)$, it follows that  
$$ 
R \circ \Phi(\omega,s,u)= \caract_{F\times (s_{0},t_{0}] \times B} RS 
$$
and then, $\Prob$-a.e. for all $t \in [0,T]$ and $z^* \in Z^*$,
\begin{eqnarray*}
\quadraVari{\int^{t}_{0}\!\! \int_{A} R \circ \Phi (r,u)\; M(dr,du),z^*}_{Z^{**},Z^*} 
& = & \caract_{F}
M((s_{0} \wedge t, t_{0} \wedge t],A\cap B)(S^{*}R^{*} z^*) 
\\
& = & \quadraVari{\int^{t}_{0}\!\! \int_{A} \Phi (r,u)\; M(dr,du),R^*z^*}_{Y^{**},Y^*}
\\
& = & \quadraVari{R^{**} \int^{t}_{0}\!\! \int_{A} \Phi (r,u)\; M(dr,du),z^*}_{Z^{**},Z^*}
\end{eqnarray*}
This extends by linearity to simple functions and by density to any $\Phi\in \Lambda(M,X,Y)$.
\end{proof}

\begin{corollary}
    Assume $X,Y$ and $Z$ are separable and reflexive Banach spaces. For $R \in \mathcal{L}(Y,Z)$ and $\Phi \in \Lambda(M,X,Y)$, $R\circ \Phi \in \Lambda(M,X,Z)$ and $\Prob$-a.e. we have, for every $t \in [0,T]$, $A\in \mathcal{A}$
$$
\int^{t}_{0}\!\! \int_{A} R \circ \Phi (r,u) M (dr, du)= R \left(\int^{t}_{0}\!\! \int_{A} \Phi (r,u) M (dr, du) \right). 
$$
\end{corollary}

A family $\Phi \in \Lambda$ that is independent of $u\in U$ can be approximated by simple families which are also independent of $u$. In the expression \eqref{eqDefSimpleFamily} for those simple families, we can assume $A_i=U$ for each $i$. This allows us to deduce the following.

\begin{corollary}
\label{lemmaPhiIndepU}
    If $\Phi\in \Lambda(M,X,Y)$  is independent of $u$ ($\Phi=\Phi(\omega,t)$) then
    $$
    \int_0^t \int_A \Phi(\omega,s)M(ds,du) = \int_0^t \Phi(\omega,s)M(ds,A),\quad A\in \calA
    $$
    where the right-hand side is a stochastic integral with respect to the cylindrical martingale $M(t,A),t\geq 0$.
\end{corollary}

\subsection{Quadratic variation of the integral}
\label{subSectQuadraVariOfIntegral}

In what follows we fix $\Phi\in \Lambda(M,X,Y)$ and denote
$$
N := \int \Phi\; dM.
$$
Given $t>s$, $A\in \calA$ and $F\in \calF_s$, it is clear that
$$
\int_F \left| N_t(A)y^* \right|^2 d\Prob  =  \int_F \int_{(s,t]\times A} \quadraVari{Q_M \Phi^*y^*,\Phi^* y^*}_{X,X^*} d\operQuadraVari{M}\, d\Prob
$$
and this implies
$$
\Exp \left[ \left. \left| N_t(A)y^* \right|^2 \,\, \right| \,\, \calF_s \right] = \Exp \left[ \left. \int_{[0,t]\times A} \quadraVari{Q_M \Phi^*y^*,\Phi^* y^*}_{X,X^*} d\operQuadraVari{M} \,\, \right| \,\, \calF_s \right].
$$
Notice that, for $A\in \calA$ and $y^* \in Y^*$ fixed,
$$
N_t(A)y^* - N_s(A)y^* = I^{\caract_{(s,t]} \Phi}_{T} (A)y^*.
$$
It follows that
$$
\Exp \left[ \left. \left| I^{\caract_{(s,t]} \Phi}_{T} (A)y^* \right|^2 \,\, \right| \,\, \calF_s  \right] =
\Exp \left[ \left. \left| N_{t} (A)y^* \right|^2 \,\, \right| \,\, \calF_s  \right] - 
\left| N_{s} (A)y^* \right|^2
$$
and consequently
\begin{equation}
\label{eqConditionalIsometry}
    \Exp \left[ \left. \left| N_{t} (A)y^* \right|^2 \,\, \right| \,\, \calF_s  \right] = \left| N_{s} (A)y^* \right|^2 + \Exp \left[ \left. 
\int_{(s,t]\times A} \quadraVari{Q_M \Phi^* y^*,\Phi^* y^*}_{X,X^*} d\operQuadraVari{M} \,\, \right| \,\, \calF_s  \right].
\end{equation}
This shows that, $\Prob$-a.e.
$$
\nu^{N}_{y^*}\left( (s,t]\times A \right)  =  \quadraVari{N(A)y^*}^t_s = \int_{(s,t]\times A} \quadraVari{Q_M \Phi^* y^*,\Phi^* y^*}_{X,X^*} d\operQuadraVari{M}.
$$
Consider $(y_n^*)$ a sequence in the unit ball of $Y^*$ such that $y_n^*\fle y^*$ (norm convergence in $Y^*$). Since $\Phi\in L(M,X,Y)$ we have 
$$
\inner{Q_M\Phi^*y_n^*}{\Phi^*y_n^*} \fle \inner{Q_M\Phi^*y^*}{\Phi^*y^*}
$$
and by Fatou's Lemma, for $0\leq s \leq t$ and $A\in \mathcal{A}$
$$
\nu_{y^*}^{N}((s,t]\times A) \leq \liminf \nu_{y_n^*}^{N}((s,t]\times A) \leq \sup \nu_{y_n^*}^{N}((s,t]\times A).
$$
This shows, according to \cite{CCFM:SPDE}*{Remark 5.6}, that the family $(\nu^{N}_{y^*})$ satisfies the sequential boundedness property. Besides, for $C=(s,t]\times A$ we have $\Prob$-a.e.
\begin{eqnarray*}
    \left(\sup_{n\in \N} \nu^{N}_{y_n^*}\right)(C) & = & \int_{C}  \sup_{n\in \N} \quadraVari{\Phi(r,u) Q_M(r,u) \Phi^*(r,u) y_n^*,y_n^*}_{Y^*,Y} \operQuadraVari{M}(dr,du) \\
    & \leq & \int_C \norm{\Phi Q_M \Phi^*}_{\calL(Y^*,Y)}d\operQuadraVari{M}.
\end{eqnarray*}
With the additional assumption
\begin{equation}
\label{eqAdditionalHip}
    \forall A\in \mathcal{A} \quad \int_{\Omega\times [0,T]\times A} \norm{\Phi Q_M \Phi^*}_{\calL(Y^*,Y)}d\mu_M < \infty
\end{equation}
according to Theorem $5.8$ and Lemma 2.2 in \cite{CCFM:SPDE},  $N$ has a unique quadratic variation
\begin{equation}
\label{eqQuadraticVariationN}
\operQuadraVari{N} (C) = \int_{C} \norm{\Phi Q_M \Phi^*}_{\calL(Y^*,Y)} d\operQuadraVari{M} \quad \Prob\textup{-a.e.}
\end{equation}

In the last calculation, we used the fact that $\Phi Q_M \Phi^*$ is symmetric and positive, so the norm can be computed as a supremum over the diagonal (see \cite{VakhaniaTarieladzeChobanyan}*{Prop III.1.2, p.p. 146}). 

\begin{theorem}
\label{theoDensities}
If $M$ satisfies \eqref{eqAdditionalHip} and $\Phi\in \Lambda(M,X,Y)$, then $N = \int \Phi\;dM$ has a unique quadratic variation given by \eqref{eqQuadraticVariationN}. In particular
$$
\frac{d\operQuadraVari{N}}{d\operQuadraVari{M}} = \norm{\Phi Q_M \Phi^*}_{\calL(Y^*,Y)}\; \Prob\textup{-a.e.}, \quad \frac{d\mu_N}{d\mu_M} = \norm{\Phi Q_M \Phi^*}_{\mathcal{L}(Y^*,Y)},\quad Q_{N} = \frac{ \Phi Q_M \Phi^* }{\norm{ \Phi Q_M \Phi^*}_{\calL(Y^*,Y)}} .
$$
\end{theorem}
\begin{proof}
It only rests to verify the identities. The first two follow directly from \eqref{eqQuadraticVariationN}, taking expectation in the second case. Now we use polarization in \eqref{eqConditionalIsometry} and get
$$
\Exp \left[ N_t(A)y^*_{1} N_t(A)y^*_{2} \, \vline \, \mathcal{F}_{s} \right] 
=  N_s(A)y^*_{1} N_s(A)y^*_{2}
+  \Exp \left(  \int_{s}^{t} \int_A \quadraVari{\Phi Q_M \Phi^{*}y^*_{1},y^*_{2}}_{Y,Y^*}\, d\operQuadraVari{M} \, \vline \, \mathcal{F}_{s} \right)  .   
$$
Therefore, with $C=(s,t]\times A$
$$
\alpha_{N}(C)(y^*_1,y^*_2) = \int_{C} \quadraVari{\Phi Q_M \Phi^*y^*_1,y^*_2}_{Y,Y^*} d\operQuadraVari{M} =
\int_C \frac{\quadraVari{\Phi Q_M \Phi^*y^*_1,y^*_2}_{Y,Y^*}}{\norm{ \Phi Q_M \Phi^*}_{\calL(Y^*,Y)}} d\operQuadraVari{N}
$$
and the third identity follows.
\end{proof}

\begin{remark}
In the Hilbert space case we can express 
$$
\nu^{N}_{g}\left( C \right) =  \int_{ C } \norm{Q_M^{1/2}\Phi^* g}_H^2 d\operQuadraVari{M},\qquad 
\operQuadraVari{N}(C) = \int_{C} \norm{\Phi Q_M^{1/2}}_{\calL(H,G)}^2 d\operQuadraVari{M} 
$$
whence
$$
\frac{d\operQuadraVari{N}}{d\operQuadraVari{M}} =\norm{\Phi Q_M^{1/2}}_{\calL(H,G)}^2\;\;\Prob-a.e.\qquad
\frac{d\mu_{N}}{d\mu_{M}} = \norm{\Phi Q_M^{1/2}}_{\calL(H,G)}^2
$$
and
\begin{equation}
    \label{eqQN}
    Q_N=\frac{\Phi Q_M\Phi^*}{\norm{\Phi Q_M \Phi^*}_{\mathcal{L}(G)}} = \frac{\Phi Q_M\Phi^*}{\norm{\Phi Q_M^{1/2}}^2_{\mathcal{L}(H,G)}}.
\end{equation}
\end{remark}

\subsection{Application: White Noise Measure on a Banach space} \label{subsectWhiteNOiseBanach}

In this section we introduce a construction of a white noise measure as an orthogonal cylindrical martingale-valued measure on a Banach space. 

Consider a white noise $W$ on $\left(\mathbb{R}_{+} \times U, \mathcal{B}\left(\mathbb{R}_{+}\right) \otimes \mathcal{B}(U),Leb \otimes \lambda\right)$. $W$ is defined at least on the cylinder sets $(s, t] \times A$, $0 \leq s \leq t$, $A \in \mathcal{A}$. Recall that 
$\mathcal{A}$ is a ring of finite $\lambda$-measure sets, $\mathcal{A} \subseteq \mathcal{B}(U)$. We define
$$
M_t(A)=W([0, t] \times A), \quad M_t(A) \overset{d}{=} N(0, t \lambda(A)).
$$
Notice that $\left\{M_t(A): t \geq 0, A \in \mathcal{A}\right\}$ is a martingale measure with respect to its natural filtration. For $A,B$ disjoint, $M_t(A)$ and $M_t(B)$ are independent, hence orthogonal. We also have
$$
\langle M(A)\rangle_t=t \lambda(A),\qquad A\in \mathcal{A}.
$$
We assume $\lambda(U_n)<\infty$ for each $n$, where $U_n \uparrow U$ and $\mathcal{B}(U_n) \subseteq \mathcal{A}$ for each $n$.

As shown in \cite{CCFM:SPDE}*{Example 3.13}, $M$ is a (real-valued) orthogonal martingale-valued measure. We can define a cylindrical orthogonal martingale-valued measure on $X^*$ by
$$
M_t(A)x^* := W([0,t]\times A)\quadraVari{x^*,l}
$$
where $l\in X^{**}=X$. It is not difficult to show that $M$ is in fact a cylindrical orthogonal martingale-valued measure on $X^*$, with intensity measures
$$
\nu_{x^*}((s,t]\times A) = (t-s)\lambda(A) \abs{\quadraVari{x^*,l}}^2
$$
and therefore, it has a unique quadratic variation given by
$$
\operQuadraVari{M} = \sup_{\norm{x^*}= 1} \nu_{x^*} = \norm{l}^2Leb\otimes \lambda.
$$
We also have, for $C=(s,t]\times A$,
$$
\alpha_M(\omega)(C)(x_1^*,x_2^*) = \tfrac{1}{4}\left( \nu_{x_1^* + x_2^*}(C) - \nu_{x_1^* - x_2^*}(C) \right) = (t-s)\lambda(A)\quadraVari{x_1^*,l}\quadraVari{x_2^*,l}
$$
and by Theorem $5.16$ in \cite{CCFM:SPDE}, $Q_M$ is the constant process
$$
\quadraVari{Q_M(\omega,r,u)x_1^*,x_2^*} = \frac{\quadraVari{x_1^*,l}\quadraVari{x_2^*,l}}{\norm{l}^2}. 
$$

Now consider $\Phi\in \Lambda(M,X,Y)$ and define $N:=\int \Phi\;dM$. Notice that
\begin{eqnarray*}
   \norm{\Phi Q_M \Phi^*}_{\calL(Y^*,Y)} &=& \sup_{\norm{y_1^*}\leq 1} \sup_{\norm{y_2^*}\leq 1} \quadraVari{\Phi Q_M \Phi^* y_1^*,y_2^*} \\
   &=& 
\sup_{\norm{y_1^*}\leq 1} \sup_{\norm{y_2^*}\leq 1} \quadraVari{Q_M \Phi^* y_1^*,\Phi^*y_2^*} \\
&=& 
\sup_{\norm{y_1^*}\leq 1} \sup_{\norm{y_2^*}\leq 1} \frac{\quadraVari{\Phi^*y_1^*,l}\quadraVari{\Phi^*y_2^*,l}}{\norm{l}^2} = \frac{\norm{l\circ \Phi^*}^2}{\norm{l}^2}.
\end{eqnarray*}
The additional hypothesis \eqref{eqAdditionalHip} becomes
$$
\forall\; A\in \mathcal{A}\quad \Exp 
\int_{0}^T \int_A \norm{ l\circ \Phi^*(\omega,t,u)}^2 dt\; \lambda(du) < \infty.
$$
With this hypothesis, Theorem \ref{theoDensities} guarantees the existence of a unique quadratic variation for $N$, that is
$$
d\operQuadraVari{N} = \frac{\norm{l\circ \Phi^*}^2}{\norm{l}^2} d\operQuadraVari{M} = \norm{ l\circ \Phi^*}^2 dt\; \lambda(du)
$$
and likewise
$$
d\mu_{N} = \frac{\norm{l \circ \Phi^*}^2}{\norm{l}^2} d\mu_M = \norm{ l\circ \Phi^*}^2 d\Prob\; dt\; \lambda(du)
$$
The same theorem gives 
$$
\quadraVari{Q_{N}y_1^*,y_2^*} = 
\frac{\quadraVari{\Phi^*y_1^*,l}\quadraVari{\Phi^*y_2^*,l}}{\norm{ l\circ \Phi^*}^2}.
$$

\subsection{Associativity of the stochastic integral}\label{subsectAssociaStochIntegral}

Throughout this section we assume that $X,Y$ are reflexive and separable Banach spaces, while $Z$ is a Banach space with separable dual.

Being the integral a new cylindrical orthogonal martingale-valued measure, we can think of integrating with respect to it. Concretely, let 
$$
N := \int \Phi dM
$$
and $\Psi\in \Lambda(N,Y,Z)$.

\begin{theorem}
\label{theoAssociatIntegral}
With the previous hypothesis we have
$$
\int \Psi\; dN = \int \Psi \circ \Phi\; dM.
$$
More precisely, for $t\in [0,T]$ and $A\in\mathcal{A}$,
\begin{equation}
\label{eqAssociativityOfStochIntegral}
    \int_0^t\int_A \Psi(s,u)N(ds,du) = \int_0^t \!\!\int_A \Psi(s,u)\circ \Phi(s,u) M(ds,du)\quad \Prob-a.e.
\end{equation}
\end{theorem}

We start by considering $\Phi$ and $\Psi$ elementary, that is
$$
\Phi = \caract_{F\times (s_0,t_0]\times B}S,\quad
\Psi = \caract_{G\times (s_1,t_1]\times C}T
$$
where $F\in \mathcal{F}_{s_0}$, $G\in\mathcal{F}_{s_1}$, $S\in \mathcal{L}(X,Y)$ and $T\in \mathcal{L}(Y,Z)$. As we will explain, it is enough to consider the cases $(s_0,t_0]\cap (s_1,t_1] = \emptyset$ and $(s_0,t_0]=(s_1,t_1]$. In the first case it is easy to see that 
$$
\int \Psi\; dN = 0 = \int \Psi\circ \Phi\; dM.
$$
In the second case, for $z^*\in Z^*$, $t\in [0,T]$ and $A\in \mathcal{A}$ we have
\begin{eqnarray*}
    \left(\int \Psi\; dN\right)(t,A)(z^*) & = & 
\caract_{G} N\left( (t\land s_0,t\land t_0],A\cap B \right)(T^*z^*) \\
& = & \caract_{F\cap G} M\left( (t\land s_0,t\land t_0],A\cap B\cap C \right)(S^*T^*z^*) \\
& = & \left(\int \Psi\circ \Phi\; dM\right)(t,A)(z^*).
\end{eqnarray*}

For $\Phi$ and $\Psi$ simple, we can can assume that the intervals $(s_i,t_i]$ of both integrands are the same, as we can always make partitions thinner and any two of these intervals are allowed to be disjoint or equal. By using the result just established for elementary families, it easily follows that, for $\Phi$ and $\Psi$ simple,
\begin{equation}
\label{eqAsociatSimple}
\left(\int \Psi\; dN\right)(t,A)(z^*) 
= \left(\int \Psi\circ \Phi\; dM\right)(t,A)(z^*).
\end{equation}

\begin{lemma}
    Given $\Phi\in \Lambda(M,X,Y)$ and $\Psi \in \Lambda(N,Y,Z)$ we have
    $\Psi\circ\Phi\in \Lambda(M,X,Z)$.
\end{lemma}
\begin{proof}
In fact, notice that:
\begin{enumerate}
     \item For each $(\omega,t,u)$, $\Psi(\omega,t,u)\circ \Phi(\omega,t,u) \in \mathcal{L}(X,Z)$
     \item From Theorem \ref{theoDensities}
     we  have, for each $(z_1^*,z_2^*)\in Z^*\times Z^*$
    $$
    \quadraVari{Q_N \Psi^*z_1^*,\Psi^*z_2^*} = \frac{\quadraVari{Q_M (\Psi \Phi)^*z_1^*,(\Psi\Phi)^*z_2^*}_{X,X^*} }{\norm{\Phi Q_M \Phi^*}_{\mathcal{L}(Y^*,Y)}}
    $$
    where the left-hand side and the norm on the right-hand side are predictable. It follows that,
    for each $(z_1^*,z_2^*)\in Z^*\times Z^*$, $\quadraVari{Q_M (\Psi \Phi)^* z_1^*,(\Psi \Phi)^*z_2^*}_{X,X^*}$ is predictable.
     \item We also have
\begin{eqnarray*}
\rho_{M,A}(\Psi\circ\Phi) &=& \sup_{\norm{z^*}_{Z^*}\leq 1} \left[\int_{\Omega\times [0,T]\times A} \quadraVari{Q_M (\Psi\circ\Phi)^* z^*,(\Psi\circ\Phi)^* z^*}_{X,X^*} d\mu_M \right]^{1/2}\\
&=& \sup_{\norm{z^*}_{Z^*}\leq 1} \left[\int_{\Omega\times [0,T]\times A} \quadraVari{\Phi Q_M \Phi^*(\Psi^*z^*),\Psi^*z^*}_{Y^{**},Y^*} d\mu_M \right]^{1/2}\\
&=& \sup_{\norm{z^*}_{Z^*}\leq 1} \left[\int_{\Omega\times [0,T]\times A} \quadraVari{Q_N \Psi^*z^*,\Psi^*z^*}_{Y^{**},Y^*} d\mu_N \right]^{1/2} = \rho_{N,A}(\Psi)<\infty.
\end{eqnarray*}
\end{enumerate}
for any $A\in \mathcal{A}$. 
Consider now $\Psi = \caract_{G\times (s_1,t_1]\times C}T$ and $\Phi\in \Lambda(M,X,Y)$. Let $(\Phi_n)$ be a sequence in $\mathcal{S}(M,X,Y)$ such that $\rho_{M,A}(\Phi_n-\Phi)\fle 0$ for each $A\in \mathcal{A}$. Notice that
\begin{equation}
\label{eqIneqRhoPhinPhi}
    \rho_{M,A}(\Psi\circ \Phi_n - \Psi\circ \Phi) = \rho_{M,A}(\Psi\circ(\Phi_n-\Phi)) \leq \norm{T} \rho_{M,A}(\Phi_n-\Phi)
\end{equation}
and therefore $\rho_{M,A}(\Psi\circ \Phi_n - \Psi\circ \Phi) \fle 0$.
Since each $\Psi\circ \Phi_n$ is simple, this shows that $\Psi\circ \Phi \in \Lambda(M,X,Z)$. By linearity, for each $\Psi \in \mathcal{S}(N,Y,Z)$ and $\Phi\in \Lambda(M,X,Y)$ we have $\Psi\circ \Phi \in \Lambda(M,X,Z)$. Finally, let $\Psi\in \Lambda(N,Y,Z)$ and $(\Psi_n) \in \mathcal{S}(N,Y,Z)$ such that $\rho_{N,A}(\Psi_n-\Psi)\fle 0$ for any $A\in\mathcal{A}$. For any $\Phi\in \Lambda(M,X,Y)$ and any $A\in\mathcal{A}$ we have
\begin{equation}
\label{eqRhoMARhoNA}
\rho_{M,A}(\Psi_n\circ \Phi - \Psi\circ \Phi) = \rho_{M,A}((\Psi_n- \Psi)\circ \Phi) = \rho_{N,A}(\Psi_n-\Psi)\fle 0.
\end{equation}
This completes the proof of the lemma.
\end{proof}\medskip

\begin{proof}[Proof of Theorem \ref{theoAssociatIntegral}]
Now let $\Psi = \caract_{G\times (s_1,t_1]\times C}T$ and $\Phi\in \Lambda(M,X,Y)$. Let $\Phi_n\in\mathcal{S}(M,X,Y)$ with $\rho_{M,A}(\Phi_n-\Phi)\fle 0$ for each $A\in\mathcal{A}$. Notice that
$$
\int_A \Psi dN_n-\int_A \Psi dN = \caract_F (N_n-N)((t\land s_1,t\land t_1],A\cap C)(T^*z^*)
$$
where $N_n-N = \int (\Phi_n-\Phi)dM$, so by \eqref{eqIsometria} and and \eqref{eqIneqRhoPhinPhi} we have
$$
\norm{\int_A \Psi dN_n-\int_A \Psi dN}_{\mathcal{L}(Z^*,\mathcal{M}_T^2)} \leq
2\norm{T^*} \rho_{M,A\cap C}(\Phi_n-\Phi)\fle 0.
$$
So with limits in $\mathcal{L}(Z^*,\mathcal{M}_T^2)$  we have, for any $A\in \mathcal{A}$,
$$
\int_A \Psi dN = \lim_{n\fle\infty} \int_A \Psi dN_n = 
\lim_{n\fle\infty} \int_A \Psi \circ \Phi_n \; dM = \int_A \Psi \circ \Phi\;dM,
$$
the last equality because $\Psi\circ\Phi_n$ is simple and converges to $\Psi\circ\Phi$. By linearity, it follows that
$$
\int_A \Psi\; dN = \int_A \Psi \circ \Phi\;dM
$$
 for all $\Phi\in \Lambda(M,X,Y)$ and $\Psi \in \mathcal{S}(N,Y,Z)$.

Finally, for $\Phi\in \Lambda(M,X,Y)$ and $\Psi \in \Lambda(N,Y,Z)$ let $\{\Psi_n\}$ a sequence in $\mathcal{S}(N,Y,Z)$ such that $\rho_{N,A}(\Psi_n-\Psi)\fle 0$ for all $A\in \mathcal{A}$. By definition, \eqref{eqIsometria} and \eqref{eqRhoMARhoNA} we have
$$
\int_A \Psi dN = \lim_{n \fle\infty} \int_A \Psi_n\; dN = 
\lim_{n\fle\infty} \int_A \Psi_n \circ \Phi\;dM = \int_A \Psi\circ \Phi\;dM.
$$
\end{proof}

\begin{remark}
    It is worth noting that we can integrate the identity operator $id:X\fle X$, as the constant elementary family $\textbf{id} = \caract_{\Omega \times [0,T] \times U}id$, since for $A\in \mathcal{A}$,
    $$
    \rho_{M,A}(\textbf{id}) = \sup_{\norm{y^*}\leq 1} \int_{\Omega\times [0,T]\times A} \quadraVari{Q_My^*,y^*}_{X,X^*} d\mu_M \leq \mu_M(\Omega\times [0,T]\times A) <\infty.
    $$
    We obtain
    $$
    I^{\textbf{id}}_t(A)y^* = M((0,t],A)id^*y^*=M_t(A)y^*.
    $$
    In fewer words
    $$
    \int dM = M,\qquad\text{or}\qquad \int_0^t \int_A M(ds,du) = M_t(A).
    $$
\end{remark}

\subsubsection{The Hilbert Space-valued case}
Consider a reflexive separable Banach space $X$ and $H,G$ Hilbert spaces. Consider $\Phi\in\Lambda(M,X,Y)$, $N:= \int \Phi\;dM$ and $\Psi\in \Lambda^2(N,H,G)$, according to notation in \cite{CCFM:SPDE}. The integral $\int \Psi\; dN$ is a classical square integrable martingale with values in $G$.

\begin{lemma}
    If $X$ is also Hilbert, then $\Psi\circ\Phi \in \Lambda^2(M,T,X,G)$.
\end{lemma}

\begin{proof}
    First notice that, for each $(\omega,t,u)$, $\Psi(\omega,t,u)\circ \Phi(\omega,t,u)\in \mathcal{HS}(X,G)$. 
    By using \eqref{eqQN} we  have, for each $g\in G$
    $$
    \quadraVari{\Psi Q_N \Psi^*g,g} = \frac{\quadraVari{\Psi\Phi Q_M (\Psi\Phi)^*g,g} }{\norm{\Phi Q_M \Phi^*}_{\mathcal{L}(H,G)}}
    $$
    where the left-hand side and the norm on the right-hand side are predictable. It follows that
    for each $g\in G$, $\Psi\Phi Q_M(\Psi\Phi)^*g$ is predictable. 

    Finally, if $\{g_n\}$ is an orthonormal basis of $G$ we have
    \begin{eqnarray*}
        \int \norm{(\Psi\Phi) Q_M^{1/2}}^2_{\mathcal{HS}(X,G)} d\mu_M &=& 
        \int \norm{Q_M^{1/2}(\Psi\Phi)^*)}^2_{\mathcal{HS}(G,X)} d\mu_M \\
        &=&  
        \int \sum_{n=1}^\infty \inner{Q_M^{1/2}(\Psi\Phi)^*g_n}{Q_M^{1/2}(\Psi\Phi)^*g_n} d\mu_M \\
        &=& 
        \int \sum_{n=1}^\infty \inner{(\Psi\Phi) Q_M(\Psi\Phi)^*g_n}{g_n}_G d\mu_M \\
        &=& 
        \int \sum_{n=1}^\infty \norm{\Phi Q_M \Phi^*} \inner{Q_N^{1/2}\Psi^*g_n}{Q_N^{1/2}\Psi^*g_n}_G d\mu_M \\
        &=& 
        \int \norm{\Phi Q_M^{1/2}}^2_{\mathcal{L}(X,H)} \norm{\Psi Q_N^{1/2}}^2_{\mathcal{HS}(H,G)}  d\mu_M \\
        &=& 
        \int \norm{\Psi Q_N^{1/2}}^2_{\mathcal{HS}(H,G)} d\mu_N < \infty.
    \end{eqnarray*}
    This completes the proof.
\end{proof}

\subsection{Cylindrical vs radonified stochastic integral}\label{subSectCylinVsRadonifiedIntegral}

Within the context of separable Hilbert spaces, a theory of vector-valued stochastic integration is introduced in \cite{CCFM:SPDE}, for square-integrable integrands with respect to a cylindrical orthogonal M.V.M. In this section we show that if an integrand $\Phi$ is stochastically integrable in the sense of the work \cite{CCFM:SPDE}, then it is also stochastically integrable in the sense of Section \ref{sectStochIntegAsACylMVM} and the corresponding integral coincides (in a sense to be explained below). 

We start by recalling the construction of the stochastic integral in \cite{CCFM:SPDE}. Let $H$ and $G$ denote two separable Hilbert spaces. For fixed $(\omega,t,u)\in \Omega\times [0,T]\times U$ and  $Q_{M} = Q_M(\omega,t,u)$, let 
$$
H_{Q_{M}}= Q_{M}^{1/2}(H),\quad \left( h,g \right)_{H_{Q_{M}}} = \left( Q_{M}^{-1/2} h, Q_{M}^{-1/2} g \right)_H
$$
where $Q_{M}^{-1/2}$ is the pseudo inverse of $Q_{M}^{1/2}$. Then $H_{Q_{M}}$ is a separable Hilbert space. 

Let  $\Phi: \Omega \times[0,T] \times U \rightarrow \mathcal{L}(H,G)$ be such that
\begin{enumerate}
\item For each $(\omega,t,u) \in \Omega\times [0,T]\times U$, $\Phi(\omega,t,u)\in \mathcal{HS}(H_{Q_{M}},G)$.
\item The mapping $ (\omega,t,u) \mapsto \Phi(\omega,t,u)h$ 
is $\mathcal{P}_{T}\otimes \mathcal{B}(U)/\mathcal{B}(G)$-measurable for every $h \in H$,
\item $\norm{\Phi}_{\Lambda^2(M)}$ is finite, where this quantity is defined by
\begin{equation}
\label{NewIntegrands}
\norm{\Phi}^2_{\Lambda^2(M)} \defeq
\int_{\Omega\times [0,T]\times U} \| \Phi(\omega,t,u)\circ Q^{1/2}_M(\omega,t,u)\|^2_{\mathcal{HS}(H,G)} \, d\mu_M.
\end{equation}
\end{enumerate}
Following Definition 6.4 in \cite{CCFM:SPDE} the collection of all $\Phi$ satisfying the conditions listed above is denoted by $\Lambda^2(M,H,G)$. This space is Hilbert when equipped with the Hilbertian norm $\norm{\cdot}_{\Lambda^2(M)}$. 

If $\Phi$ is of the simple form \eqref{eqDefSimpleFamily} with each $S_i \in \mathcal{HS}(H,G)$, then it is evident that $\Phi \in \Lambda^2(M,H,G)$. Moreover, in such a case the radonified stochastic integral is defined as 
\begin{equation}
\label{NewDefIntSimpleIntegrand}
R_t(\Phi)  = 
\sum_{i=1}^{n}  \caract_{F_i}(\omega) Y_{i}(t)
\end{equation}
where $Y_{i}$ is the square integrable martingale taking values in $G$ and satisfying
\begin{equation} 
\label{eqDefRadonProcessY}
\left( Y_{i}(t),g \right)_{G} = M((s_i \land t, t_i\land t], A_i) (S_{i}^*g)
\end{equation}
according to Theorem 2.1 of \cite{AlvaradoFonseca:2021}. In Theorem 6.7 in \cite{CCFM:SPDE} it is shown that 
$(R_{t}(\Phi): t \in [0,T]) \in \mathcal{M}_{T}^{2}(G)$ and for each $t \in [0,T]$, we have $\Exp (R_{t}(\Phi))=0$ and also       
\begin{equation}
\label{eqItoIsometryRadonifiedIntegral}
\Exp \left[ \norm{R_{t}(\Phi)}_{G}^{2} \right] = \norm{\Phi}^2_{\Lambda^2(M)}.
\end{equation}
In particular, the mapping $R: \mathcal{S}\rightarrow \mathcal{M}^{2}_{T}(G)$, $\Phi \mapsto R=(R_{t}(\Phi): t \in [0,T])$ is linear continuous. Now, since each $\Phi \in \Lambda^2(M,H,G)$ can be approximated in the norm $\norm{\cdot}_{\Lambda^2(M)}$ by a sequence of simple processes, then a standard approximation argument shows that the mapping $R: \mathcal{S}\rightarrow \mathcal{M}^{2}_{T}(G)$ has a linear continuous extension $R: \Lambda^2(M,H,G) \rightarrow \mathcal{M}^{2}_{T}(G)$ such that $R=(R_{t}(\Phi): t \in [0,T])$ satisfies $\Exp (R_{t}(\Phi))=0$ and 
\eqref{eqItoIsometryRadonifiedIntegral}. 

For every $t \in [0,T]$, denote $R_{t}(\Phi)$ by $\displaystyle (R)\int_0^t\!\! \int_U \Phi(r,u)M(dr,du)$. For any $A \in \calA$ it is clear that $\mathbbm{1}_{A} \Phi \in \Lambda^2(M,T)$, then we define 
$$ (R)\int_0^t\!\! \int_A \Phi(r,u)M(dr,du) \defeq R_{t}(\mathbbm{1}_{A} \Phi). $$
As it was the case for the cylindrical integral of Section \ref{sectStochIntegAsACylMVM}, it is not difficult to check that 
$$ \left( (R)\int_0^t\!\! \int_A \Phi(r,u)M(dr,du): t \in [0,T], A \in \calA \right)$$
is an $H$-valued orthogonal martingale-valued measure (see Definition 3.3 in \cite{CCFM:SPDE}).  

The relationship between the radonified integral and the non-radonifying integral is explained in the next result. 

\begin{proposition}
\label{Propcompatibilityofintegrals}
Let $\Phi \in \Lambda^2(M,H,G)$. Then $\Phi \in \Lambda(M,H,G)$ and for every $t \in [0,T]$, $A \in \calA$ and $g \in G$, 
\begin{equation}\label{eqCompatibilityOfIntegrals}
\inner{(R)\int_0^t\!\! \int_A \Phi(r,u)M(dr,du)}{g}_{G} = \int_0^t\!\! \int_A \Phi(r,u)M(dr,du)(g).     
\end{equation}
\end{proposition}
\begin{proof}
First, observe that for every $n \in \N$ we have
\begin{eqnarray*}
\rho_{M,U_n}(\Phi)^2  & \leq &  \int_{\Omega\times [0,T]\times U_n} \sup_{\norm{g}_{G}\leq 1} \inner{Q_M \Phi^* g}{\Phi^* g}_{H} d\mu_{M,n} \\
& \leq  & \int_{\Omega\times [0,T]\times U} \| \Phi Q^{1/2}_M\|^2_{\mathcal{HS}(H,G)} \, d\mu_M = \norm{\Phi}^2_{\Lambda^2(M)}.  
\end{eqnarray*}
Hence  $d_{M}(\Phi) \leq \norm{\Phi}_{\Lambda^2(M)}$. This shows $\Phi \in \Lambda(M,H,G)$. 

Now, observe that if $\Phi$ is of the simple form  \eqref{eqDefSimpleFamily} with each $S_i \in \mathcal{HS}(H,G)$,  by \eqref{eqDefiIntegSimpleProce}, 
\eqref{NewDefIntSimpleIntegrand} and \eqref{eqDefRadonProcessY} one can conclude that \eqref{eqCompatibilityOfIntegrals} holds. 

Finally, if $\Phi \in \Lambda^2(M,H,G)$ and $(\Phi_{n}: n \in \N)$ is a sequence of simple processes converging to $\Phi$ in the norm $\norm{\cdot}_{\Lambda^2(M)}$, then $(\Phi_{n}: n \in \N)$ also converges to $\Phi$ in $d_{M}$. By a standard approximation argument using \eqref{eqIsometria} and  \eqref{eqItoIsometryRadonifiedIntegral} we conclude  \eqref{eqCompatibilityOfIntegrals} holds for $\Phi$. 
\end{proof}

% \begin{example}
% {\color{red} Ejemplo de que $L^2$ no es cerrado en $\Lambda$}   
% \end{example}

% We denote
% $$
% \int \Phi dM = \int \Phi(t,u) M(dt,du) := I^{\Phi},\qquad 
% \int_0^t \int_A \Phi(s,u) M(ds,du) := I_t^{\Phi}(A).
% $$

\begin{remark}
From now on if $\Phi \in \Lambda^2(M,H,G)$, due to (\ref{eqCompatibilityOfIntegrals}), we will omit the $(R)$ from the notation of the radonified integral.  

\end{remark}

%%%%%%%%%%%%%%%%% White Noise %%%%%%%%%%%%%%%%%%%%%%

\section{The martingale representation theorem in the continuous paths setting}\label{sectionMartinRepreTheorem}

In this section we pursue the establishment of a continuous path version for the martingale representation theorem for cylindrical martingale-valued measures and within the context of separable Banach spaces. To do this, we first introduce a cylindrical analogue for the white noise measure. Later we prove a L\'evy's characterization theorem for such process. Our final step is to prove a martingale representation theorem.

\subsection{Cylindrical White Noise Measure}\label{subSectCyliWhiteNoise}

Let $H$ be a separable Hilbert space. Let $Q$ be a bounded, non-negative, self-adjoint operator on $H$ so that for each $h \in H$, $Qh$ can be represented (convergence in $H$) as 
\begin{equation}\label{representationCovariaQDefiCylinWhineNoiseMeasure}
Q h=  \sum_{n=1}^{\infty} \lambda_{n} \inner{h_{n}}{h}_{H}h_{n},   
\end{equation}
where $(\lambda_{n}: n\in \N)$ is a bounded sequence of non-negative numbers and $(h_{n}: n \in \N)$ is an orthonormal set in $H$ (both sequences independent of $h$).
Let $(W_{n} : n \in \N)$, be a sequence of independent white noise measures on $\left(\mathbb{R}_{+} \times U, \mathcal{B}\left(\mathbb{R}_{+}\right) \otimes \mathcal{B}(U),Leb \otimes \lambda\right)$, that is, for each $A \in \mathcal{A}$, the processes ($W_n(t,A)$) are independent. In addition, for each $t \geq 0$, if $A$ and $B$ are disjoint, then $W_n(t,A)$ and $W_n(t, B)$ are independent. 

A \emph{cylindrical $Q$-white noise measure} in $H$ based on $\left(\mathbb{R}_{+} \times U, \mathcal{B}\left(\mathbb{R}_{+}\right) \otimes \mathcal{B}(U),Leb \otimes \lambda\right)$ is defined for every $h \in H$, $t \geq 0$ and $A \in \mathcal{B}(U)$ as
\begin{equation}
\label{eqDefWN}
W(t,A)h= \sum_{n=1}^{\infty} \lambda_{n}^{1/2} W_{n}(t,A)\inner{h_{n}}{h}_{H}.
\end{equation} 
It is easy to check that $W$ is an orthogonal cylindrical martingale-valued measure on $H$. It is enough to consider the properties of the white noise measures $(W_n)$ (see Example 3.13 of \cite{CCFM:SPDE}) and a simple computation of the second moment of $W(t,A)h$ for $t \geq 0$, $A \in \mathcal{B}(U)$ and $h \in H$.   

Observe  that by the independence of the $W_{n}$'s we have
\begin{equation*}
\quadraVari{W(A)h}_{t} 
=  \sum_{n=1}^{\infty} \lambda_{n} \inner{h_{n}}{h}^{2}_{H} \quadraVari{W_{n}(A)}_{t} \\
=  t \lambda(A) \sum_{n=1}^{\infty} \lambda_{n} \inner{h_{n}}{h}^{2}_{H}.      
\end{equation*}

\begin{remark}
Observe that if $Q=I$ (the identity operator on $H$), then $\lambda_{n}=1$ for every $n \in \N$ and we can choose $(h_{n}: n \in \N)$ to be an orthonormal basis. Hence $\sum_{n=1}^{\infty} \lambda_{n}^{2} \inner{h_{n}}{h}^{2}_{H} =\norm{h}_{H}^{2}  $. So we have $\quadraVari{W(A)h}_{t}= t \lambda(A) \norm{h}_{H}^{2}$ and the cylindrical white noise is ``standard''. 
\end{remark}

We have shown that the intensity measures of $W$ are given by 
$$ \nu_{h}(t,A)= t \lambda(A) \sum_{n=1}^{\infty} \lambda_{n} \inner{h_{n}}{h}^{2}_{H}. $$

Hence, the quadratic variation of $W$ is given by
$$ \operQuadraVari{W}(t,A)=\sup_{\norm{h}=1} \nu_{h}(t,A) = t \lambda(A) \sup_{n \geq 1} \lambda_{n} =  t \lambda(A) \lambda_{1} = t \lambda(A) \norm{Q}_{\mathcal{L}(H)}. $$
Therefore, 
\begin{equation}\label{eqQuadraVariaCylindricalWhiteNoise}
    \operQuadraVari{W}= \norm{Q}_{\mathcal{L}(H)} \, \mbox{Leb} \otimes \lambda. 
\end{equation}
We have also, for $C=(s,t] \times A$, 
\begin{equation*}
\alpha_{W}(\omega)(C)(h,h) 
=  \frac{1}{4} \nu_{2h}(C)  
=   (t-s) \lambda(A) \sum_{n=1}^{\infty} \lambda_{n} \inner{h_{n}}{h}^{2}_{H}.
\end{equation*}
So by Theorem $5.16$ in \cite{CCFM:SPDE}, we have
$$ \int_{C} \inner{Q_{W}(\omega,r,u)h}{h}_{H} \operQuadraVari{W}(dr,du)= (t-s) \lambda(A) \sum_{n=1}^{\infty} \lambda_{n} \inner{h_{n}}{h}^{2}_{H}.$$
Hence
$$ \inner{Q_{W}(\omega,r,u)h}{h}_{H}  = \frac{1}{\norm{Q}_{\mathcal{L}(H)}} \sum_{n=1}^{\infty} \lambda_{n} \inner{h_{n}}{h}^{2}_{H} = \frac{\inner{Q h}{h}_{H}}{\norm{Q}_{\mathcal{L}(H)}} . $$
Then, we have $Q_{W}$ is the constant process
\begin{equation}\label{eqOperatorQWofCylindWhiteNoise}
Q_{W}(\omega,r,u) = \frac{Q}{\norm{Q}_{\mathcal{L}(H)}}. 
\end{equation}

In particular, if $\norm{Q}_{\calL(H)} = 1$ we have $Q_W (\omega,r,u)= Q$.

\begin{remark}\label{remarkQtraceclass}
    Assume
\begin{equation}\label{eqSumLambdaN}
       \sum_{n=1}^{\infty} \lambda_n < \infty. 
\end{equation}
    Then \eqref{eqDefWN} defines an $H$-valued white noise. To see this,  observe that
 $$  \sum_{n=1}^\infty \Exp \norm{\lambda_{n}^{1/2} W_{n}(t,A)h_{n}}_{H}^2 = \sum_{n=1}^\infty \lambda_n \Exp  \abs{W_n(t,A)}^2 
    = t\lambda(A) \sum_{n=1}^\infty \lambda_n < \infty. $$     
    % \begin{eqnarray*}
    %   \sum_{n=1}^\infty \Exp \norm{W(t,A)h_{n}}_{H}^2    
    % % & = &  \sum_{k=1}^\infty \sum_{n=1}^\infty \lambda_n  \Exp \abs{W_n(t,A)}^2 \inner{h_n}{h_k}_H^2 \\
    % &= &  \sum_{n=1}^\infty \lambda_n \Exp  \abs{W_n(t,A)}^2 
    % = t\lambda(A) \sum_{n=1}^\infty \lambda_n < \infty.
    % \end{eqnarray*}
Hence, the random series 
$$ W(t,A) \defeq \sum_{n=1}^{\infty} \lambda_{n}^{1/2} W_{n}(t,A) h_{n}, $$
is absolutely summable in $L^{2}(\Omega,\mathcal{F}, \Prob;H)$. Moreover, since the sequence $( W_{n})$ is independent, the series also converges $\Prob$-a.e., this can be concluded as an application of the It\^{o}-Nisio theorem. It is worth to mention that the condition \eqref{eqSumLambdaN} is satisfied if $Q$ is a trace-class operator, and in such a case we can choose the sequence $(\lambda_{n})$ so that $\mbox{Trace}(Q) = \sum_{n=1}^{\infty} \lambda_n $. Hence, by the calculation above we have
$$ \Exp \norm{W(t,A)}_H^2 = t\lambda(A) \mbox{Trace}(Q).$$
\end{remark}

%\newpage 

\subsection{L\'evy's Characterization}\label{subSectLevyCharact}

In this section we prove a version of the L\'evy's characterization theorem for the  cylindrical $Q$-white noise measure introduced in the previous section. 

\begin{theorem}[L\'evy's Characterization]\label{theoLevyCharacterization}
    Let $M = ( M(t,A) :  t \geq 0, A \in \mathcal{A} )$ be a cylindrical orthogonal martingale valued measure with continuous paths.  Assume that $M$ has operator quadratic variation $\Gamma_M = Q (\textup{Leb}\otimes \lambda)$, where $\lambda$ is a $\sigma$-finite measure on $U$, which is finite on the sets in $\mathcal{A}$, and $Q$ is a bounded, non-negative, self-adjoint operator on $H$ which has representation \eqref{representationCovariaQDefiCylinWhineNoiseMeasure}.  Then, $M$ is a cylindrical $Q$-white noise measure in $H$ based on $\left(\mathbb{R}_{+} \times U, \mathcal{B}\left(\mathbb{R}_{+}\right) \otimes \mathcal{B}(U),Leb \otimes \lambda\right)$.
\end{theorem}
\begin{proof}
    We complete the orthonormal family $\{ h_n \}$ to form a base $\{ h_n \} \cup \{ k_n \}$, then for any $h \in H$, $t \geq 0$, $A \in A$ we have
\begin{align*}
    M(t,A)h &= M(t, A) \left( \sum_{n =1}^{\infty} (h_n, h)_H h_n +  \sum_{n =1}^{\infty} (k_n, h)_H k_n \right) \\
    & = \sum_{n =1}^{\infty} (h_n, h)_H M(t, A)h_n + \sum_{n =1}^{\infty} (k_n, h)_H M(t, A)k_n \\
    & = \sum_{n =1}^{\infty} \lambda_n^{\frac{1}{2}}(h_n, h)_H [\lambda_n^{-\frac{1}{2}} M(t, A)h_n ].
\end{align*}
The second sum in the right hand side is zero, since its quadratic variation is zero. In fact,
\begin{align*}
    \quadraVari{M(A)k_n}_t & = \nu_{k_n}([0,t] \times A) = \alpha_M ([0,t] \times A) (k_n, k_n) \\ 
    &= \inner{\Gamma_M ([0,t] \times A)k_n}{k_n}_H  = t \lambda(A) \inner{Q k_n}{k_n}_H = 0.
\end{align*}

Let $W_n(t, A) = \lambda_n^{-\frac{1}{2}} M(t, A)h_n$.  We must prove that the $W_n$ are independent white noise measures, each with intensitiy measure $\nu_n = \textup{Leb} \otimes \lambda$.  With the same computations as above, we have
\begin{align*}
    \quadraVari{ W_n(A) }_t & = \quadraVari{ \lambda_n^{-\frac{1}{2}} M(A)h_n }_t = \lambda_n^{-1} \quadraVari{ M(A)h_n }_t \\
    & = \lambda_n^{-1} \alpha_M ([0,t] \times A) (h_n, h_n) = \lambda_n^{-1} t \lambda(A) \inner{Qh_n}{h_n}_H \\
    & = \lambda_n^{-1} t \lambda(A) \lambda_n = t \lambda(A)
\end{align*}

By Proposition II-3 in \cite{KarouiMeleard:1990} we have that each $W_n$ is a white noise with intensity $\textup{Leb} \otimes \lambda$.  To prove independence, it is enough to show that the $W_n$ are orthogonal. We have

\begin{align*}
    \langle  W_m(A) , W_n(A) \rangle_t & = \langle  \lambda_m^{-\frac{1}{2}} M(t, A)h_m , \lambda_n^{-\frac{1}{2}} M(t, A)h_n \rangle_t \\
    & = \lambda_m^{-1}\lambda_n^{-1} \langle   M(t, A)h_m ,  M(t, A)h_n \rangle_t \\
    & = \lambda_m^{-1}\lambda_n^{-1} \alpha_M([0,t]\times A)(h_m, h_n) \\
    & = \lambda_m^{-1}\lambda_n^{-1} t \lambda(A) \inner{Qh_m}{h_n}_H = 0.
\end{align*}
    Which proves orthogonality and thus, independence, since each $W_{n}(t,A)$ is Gaussian.
\end{proof}

\subsection{The martingale representation theorem}\label{subsectMartingaleRepresentation}

%At a difference of Theorem \ref{martingalerepresentationtheorem}, in this case  

In this section we prove the following version of the martingale representation theorem. Our objective is to show that a C.M.V.M. with continuous paths with some given operator-valued quadratic variation of an integral form, can be expressed as the stochastic integral with respect to a cylindrical white noise measure. Let $H$ be a separable Hilbert space and $X$ be a separable reflexive Banach space. Denote by $\mathcal{K}^{+}(X)$ the collection of all the non-negative symmetric compact operators from $X^*$ to $X$. Notice the particular case of a Hilbert space $H$, in which $\mathcal{K}^{+}(H)$ is the space of non-negative, self-adjoint compact operators.

\begin{theorem}
[Martingale representation theorem]
\label{martingalerepresentationtheoremContinuousCase}
Let $M=(M(t, A): t \geq 0, A \in \calA)$ be an orthogonal C.M.V.M. on $X^*$ with continuous paths and which possesses a unique quadratic variation.  
Assume that the operator-valued quadratic variation of $M$ is given by 
\begin{equation}\label{eqOperatorQuaVariMRepresTheoremContinuous}
 \Gamma_{M}([0,t] \times A) = \int_0^t \int_A\; \Phi (r,u) Q\Phi^*(r,u)\; \lambda(du) dr, 
\end{equation}
where $Q$ is a bounded, non-negative, self-adjoint operator on $H$ with  representation  \eqref{representationCovariaQDefiCylinWhineNoiseMeasure}, $\lambda$ is a $\sigma$-finite measure on $U$, which is finite on the sets in $\mathcal{A}$, and $\Phi: \Omega \times[0,T] \times U \rightarrow \mathcal{L}(H,X)$ is $ (\omega,t,u) \mapsto \Phi(\omega,t,u)h$ 
is $\mathcal{P}_{T}\otimes \mathcal{B}(U)/\mathcal{B}(X)$-measurable for every $h \in H$ and such that for $\Psi(\omega,t,u)=\Phi(\omega,t,u) Q^{1/2}$, we have $\Psi(\omega,t,u) \Psi^{*}(\omega,t,u) \in \mathcal{K}^{+}(X)$.

Then there exists a 
cylindrical $Q$-white noise measure $W=(W(t, A): t \geq 0, A \in \calA)$ on $H$ based on $\left(\mathbb{R}_{+} \times U, \mathcal{B}\left(\mathbb{R}_{+}\right) \otimes \mathcal{B}(U),Leb \otimes \lambda\right)$  defined with respect to an extended probability space $(\Omega \times \widetilde{\Omega}, \mathcal{F} \times \widetilde{\mathcal{F}}, \Prob \times \widetilde{\Prob})$ adapted to the filtration $\{\mathcal{F}_{t} \times \widetilde{\mathcal{F}}_{t}\}$, such that 
 $$
M(t,A)(\omega)= \int_0^t\int_A \Phi(\omega,r,u)W(dr,du)(\omega, \widetilde{\omega}),
$$ 
where the last equality must be interpreted as cylindrical random variables.
\end{theorem} 

\begin{remark}\label{remaTraceClassQ} \hfill 
\begin{enumerate}
    \item If $Q \in \mathcal{K}^{+}(H)$, then $Q$ has a representation of the form \eqref{representationCovariaQDefiCylinWhineNoiseMeasure} with $\lambda_{n} \searrow 0$ and where the sum converges in the operator norm. In such case by the ideal property of compact operators we have  $\Psi(\omega,t,u) \Psi^{*}(\omega,t,u) \in \mathcal{K}^{+}(X)$. 
    \item If $Q \in \mathcal{L}^{+}_{1}(H)$ (in particular this implies $Q \in \mathcal{K}^{+}(H)$), by Remark \ref{remarkQtraceclass}, $W$ can be chosen to be an $H$-valued white noise measure.
    %\item {\color{red} Explorar m\'as opciones de radonificaci\'on}
\end{enumerate}
\end{remark}

\begin{proof}  We separate the proof in two steps. 

\emph{Step 1:
Assume $X=G$, where $G$ is a separable Hilbert space.}

First, observe that $\Psi\Psi^*$ satisfies all the conditions in Proposition \ref{propCanonicalDecompoRandomCompact}, hence has the representation  
\begin{equation}\label{eqSpectralDecompoPsiAndAdjointPsiContinuous}
    \Psi(\omega,r,u)\Psi^*(\omega,r,u)g = \sum_{n=1}^\infty \lambda_n(\omega,r,u) \inner{\varphi_n(\omega,r,u)}{g}_G\varphi_n(\omega,r,u)
\end{equation}
    where $\lambda_{n}$ and $\varphi_{n}$ satisfy \eqref{eqNormFunctionsG} and 
\eqref{eqOrthogonalFunctionsG}. 
%Moreover, as part of the definition of $\lambda_{1}$ and $g_{1}$ recall that 
% \begin{equation}
% \norm{\Phi(r,u)Q\Phi^*(r,u)}_{\calL(G)}=\norm{\Psi\Psi^*}_{\calL(G)} = \lambda_1(\omega,t,u).
% \end{equation}
    Let $(h_n)_{n\in\N}$ be a ONB in $H$ which extends the orthonormal system associated to $Q$ in the representation \eqref{representationCovariaQDefiCylinWhineNoiseMeasure}. Define $V(\omega,r,u):H\fle G$ by
    \begin{equation*}\label{eqOperatorVRepresentationTheoremContinuous}
    V(\omega,r,u)h = \sum_{n=1}^\infty \inner{h}{h_n}_H\varphi_n(\omega,r,u).
    \end{equation*}

    Notice that
    $$
    V^*(\omega,r,u)g = \sum_{n=1}^\infty \inner{\varphi_n(\omega,r,u)}{g}_G h_n
    $$
and $\norm{V(\omega,r,u)} = \norm{V^*(\omega,r,u)} = 1$.
    The $\mathcal{P}_T\otimes \calB(U)/\calB(G)$-measurable process
    $$
    \Pi(\omega,r,u) = V(\omega,r,u)V^*(\omega,r,u)
    $$
    is an orthogonal projection on $\mbox{Im}(\Psi(\omega,r,u)\Psi^*(\omega,r,u))$. Therefore
    \begin{eqnarray*}
        M(t,A) &=& \int_0^t\int_A \left(\Pi(\omega,r,u)+\Pi^{\perp}(\omega,r,u)\right) M(dr,du)\\
        &=& \int_0^t\int_A \Pi(\omega,r,u) M(dr,du) + \int_0^t\int_A \Pi^{\perp}(\omega,r,u)  M(dr,du)\\
        &=& M_1(t,A)+M_2(t,A).
    \end{eqnarray*}
Observe that the above stochastic integrals exist according to Theorem \ref{theoCondL2LocInteg}. Moreover,
    $$
     \Gamma_{M_{2}}([0,t] \times A) = \int_0^t\int_A \Pi^{\perp}(r,u)\Psi(r,u)\Psi^*(r,u)\Pi^{\perp}(r,u)\lambda(dr,du) = 0
    $$
    and consequently, thanks to the associativity of the integral (Theorem \ref{theoAssociatIntegral}), and Theorem \ref{theoCondL2LocInteg} (for the existence of the stochastic integrals),
\begin{equation*}\label{eqMAsIntegralOfNContinuous}
        M(t,A) = \int_0^t\int_A V(r,u) N(dr,du),
    \end{equation*}
    where
    \begin{equation}
    \label{eqMAsIntegralOfMContinuous}
    N(t,A):=\int_0^t\int_A V^*(r,u)M(dr,du)
    \end{equation}
    is an orthogonal CMVM on $H$ (notice that we cannot have a radonified integral here), whose operator-valued quadratic variation is given by
    $$
    \Gamma_{N}((s,t]\times A) = \int_0^t\int_A V^*(r,u)\Psi(r,u)\Psi^*(r,u)V(r,u)\lambda(dr,du).
    $$

Then, for every $h \in H$ we have  
 \begin{eqnarray*}
N(t,A)h & = & \int_{0}^{t} \int_A \quadraVari{V^*(r,u),h} M(dt,du) \\
& = &  \sum_{n=1}^\infty (h,h_n) \int_0^t\int_A \varphi_n(r,u)M(dr,du) \\
& = & \sum_{n=1}^\infty (h,h_n)\eta_n(t,A)
\end{eqnarray*}
   where for each $n \in \N$,  
\begin{equation}
\label{etaprocessdefin}
\eta_n(t,A) = \int_0^t\int_A \varphi_n(r,u)M(dr,du).
\end{equation}   
   is a real-valued  orthogonal martingale-valued measure (in the sense of \cite{Walsh:1986}). 
Now, we can write 
\begin{equation}\label{eqMeasureNAsRandomFourier}
        N(t,A) = \sum_{n=1}^\infty \eta_n(t,A)h_n 
\end{equation}
    as a random series of orthogonal MVM. Observe that by \eqref{eqMeasureNAsRandomFourier}, the decomposition of $\Psi \Psi^{*}$ in \eqref{eqSpectralDecompoPsiAndAdjointPsiContinuous} and \eqref{eqOperatorQuaVariMRepresTheoremContinuous}, we have
 \begin{eqnarray*}
 \Exp \left[ \abs{\eta_{n}(t,A)}^{2} \right]
 & = &  \Exp \left[ \abs{N(t,A)h_{n}}^{2} \right] \\
& = &  \Exp \left[ \abs{ \int_{0}^{t} \int_{A} \langle V^{*}(r,u), h_{n} \rangle M(dr,du) }^{2} \right] \\
& = &  \Exp \left[ \abs{ \int_{0}^{t} \int_{A} \inner{\varphi_{n}(r,u)}{\cdot}_{G} M(dr,du) }^{2} \right] \\
& = &  \Exp \int_{0}^{t} \int_{A} \lambda_{n}(r,u) \lambda(dr,du) 
\end{eqnarray*}
Therefore, $\eta_n$ has an intensity measure given by 
\begin{equation}\label{eqIntensityMeasureOfEtaN}
    \nu_{N,n}(t,A) = \int_0^t\int_A  \lambda_{n}(r,u) \lambda(dr,du).
   \end{equation}

Let 
\begin{equation}
\label{deltagammadefin}
   \delta_{n}(\omega,r,u) =
\begin{cases}
\lambda_{n}^{-1/2}(\omega,r,u), & \mbox{ if } \lambda_{n}(\omega,r,u)>0, \\
0, & \mbox{ if } \lambda_{n}(\omega,r,u)=0, \\
\end{cases}  \quad \gamma_{n}(\omega,r,u) =
\begin{cases}
0, & \mbox{ if } \lambda_{n}(\omega,r,u)>0, \\
1, & \mbox{ if } \lambda_{n}(\omega,r,u)=0. \\
\end{cases}
\end{equation}

We define now
$$
Y_n(t,A)=\int_0^t \int_A \delta_n(s,u) \, \eta_n(ds, du)+\int_0^t \int_A \gamma_n (s,u)\, w_n(ds,du),
$$
where $w_n$ are independent white noise measures on $\left(\mathbb{R}_{+} \times U, \mathcal{B}\left(\mathbb{R}_{+}\right) \otimes \mathcal{B}(U),Leb \otimes \lambda\right)$ defined on a new probability space $(\Omega', \mathcal{F}', \{\mathcal{F}_t'\}_{t \geq 0}, \Prob')$ and $\eta_n$ is defined in (\ref{etaprocessdefin}). 
We extend processes defined on $\Omega$ or $\Omega'$ to the product space $\Omega \times \Omega'$, for example by $M(\omega, \omega', t,  A) = M(\omega, t, A)$.
Then, if we use (\ref{eqQuadraVariaCylindricalWhiteNoise}),  (\ref{eqIntensityMeasureOfEtaN}) and (\ref{deltagammadefin}), we have that for each $A \in \mathcal{A}$ one has that
$$
\quadraVari{Y_n(A), Y_m(A)}_t=\delta_{m,n}\,t\, \lambda(A), \quad t \geq 0.
$$
Observe that by the above calculation and Proposition II-3 in \cite{KarouiMeleard:1990} we have that the $Y_{n}$ are orthogonal (hence independent) white noise measures on $\left(\mathbb{R}_{+} \times U, \mathcal{B}\left(\mathbb{R}_{+}\right) \otimes \mathcal{B}(U),Leb \otimes \lambda\right)$.

Now, for each $k \in H$, we consider the sum
$$\widehat{W}(t,A)k=\sum_{n=1}^\infty Y_n(t,A) (h_n,k),$$
which is a cylindrical $I$-white noise measure in $H$ on $\left(\mathbb{R}_{+} \times U, \mathcal{B}\left(\mathbb{R}_{+}\right) \otimes \mathcal{B}(U),Leb \otimes \lambda\right)$. Note that
$$
\int_0^t \int_A \lambda_n^{1/2}(r,u) Y_n(dr,du)=\int_0^t \int_A \lambda_n^{1/2}(r,u) \delta_n(r,u) \eta_n(dr,du)=\eta_n(t,A).
$$
Hence, if we define
\begin{equation*}
 \Lambda^{1/2}(\omega,r,u)h = \sum_{n=1}^{\infty} \lambda_{n}^{1/2}(\omega,r,u)\inner{h_n}{h}_{H} h_{n},    
\end{equation*}
then we have
\begin{eqnarray*}
N(t,A) &=& \sum_{n=1}^\infty \eta_n(t,A)h_n=\sum_{n=1}^\infty \int_0^t \int_A \lambda_n^{1/2}(r,u)\, Y_n(dr,du)\\
&=& \int_0^t \int_A \Lambda^{1/2} (r,u) \widehat{W}(dr,du),
\end{eqnarray*}
Thus, it follows that 
\begin{equation}
\label{CMVMhatWequality}
M(t,A)=\int_0^t \int_A V(r,u)\, N(dr,du)=\int_0^t \int_A V(r,u)\Lambda^{1/2}(r,u) \widehat{W}(dr,du),
\end{equation}
with the help of \eqref{eqMAsIntegralOfNContinuous}. Now, if we consider another filtered probability space $(\Omega'', \mathcal{F}'', \{\mathcal{F}_t''\}_{t \geq 0}, \Prob'')$ and $\widetilde{W}$ a cylindrical $I$-white noise measure in $H$ on $\left(\mathbb{R}_{+} \times U, \mathcal{B}\left(\mathbb{R}_{+}\right) \otimes \mathcal{B}(U),Leb \otimes \lambda\right)$. We extend the processes to the product space $\Omega\times \Omega' \times \Omega''$, for example by $M(\omega, \omega', \omega'', t,  A) = M(\omega, \omega', t, A) = M(\omega, t, A)$. 
We define
$$
W(t,A)=\int_0^t \int_A Q^{1/2}(r,u) J(r,u)\, \widehat{W}(dr,du)+\int_0^t \int_A Q^{1/2}(r,u) K(r,u)\, \widetilde{W}(dr,du),
$$
with $K(\omega,r,u)=(J(\omega,r,u) J^{\ast}(\omega,r,u))^{\perp}$, where $J(\omega,r,u):(\mbox{ker} \widehat{\Psi}(\omega,r,u) )^{\perp} \rightarrow (\mbox{ker} \Psi(\omega,r,u) )^{\perp}$ is a $\mathcal{P}_T \otimes \calB(U)$-measurable process, which is obtained as an application of Lemma 2.10 in \cite{GawaMand:2010} such that 
\begin{equation}
\label{HatPsiJprocessequality}
 \widehat{\Psi}(\omega,r,u) = \Psi(\omega,r,u)J(\omega,r,u),
 \end{equation}
and such that $J(\omega,r,u) J^{\ast}(\omega,r,u)$ is an orthogonal projection onto $(\mbox{ker} \Psi(\omega,r,u) )^{\perp} $.

%(see (\ref{PsihatPsiequality}), (\ref{HatPsiJprocessequality}))
Then,
\begin{eqnarray*}
\Gamma_W \left([0,t] \times A\right)&=&\int_0^t \int_A \Bigl(Q^{1/2} J(r,u) J^{*}(r,u) Q^{1/2}+Q^{1/2}K(r,u)Q^{1/2} \Bigr)\, dr \,\lambda(du)\\
&=& Q\, t\, \lambda(A).
\end{eqnarray*}
 Thus, as an application of Theorem \ref{theoLevyCharacterization}, $W$ is a cylindrical $Q$-white noise measure in $H$ on $\left(\mathbb{R}_{+} \times U, \mathcal{B}\left(\mathbb{R}_{+}\right) \otimes \mathcal{B}(U),Leb \otimes \lambda\right)$. 
 
Finally, by \eqref{eqOperatorQuaVariMRepresTheoremContinuous} and  Theorem \ref{theoCondL2LocInteg} (recall that $U_{n} \in \calA$ for every $n \in \N$) we conclude that $\Phi \in \Lambda(W,H,G)$.
 Furthermore, if we use the fact that $K$ is the projection onto $\mbox{ker} \Psi(\omega,r,u)$ and then by \eqref{HatPsiJprocessequality}, we  conclude that
\begin{eqnarray*}
\int_0^t \int_A \Phi(r,u)\, W(dr,du) &=&  \int_0^t \int_A \Psi(r,u) J(r,u)\, \widehat{W}(dr,du) + \int_0^t \int_A \Psi(r,u) K(r,u)\, \widetilde{W}(dr,du)\\
&=&   \int_0^t \int_A \widehat{\Psi}(r,u) \, \widehat{W}(dr,du)\\
&=& M(t,A),
\end{eqnarray*}
where the last equality is a direct consequence of (\ref{CMVMhatWequality}).

\emph{Step 2: $X$ is a separable, reflexive Banach space}.

%\textbf{Generalization to separable Banach spaces}

 Following the ideas in the proof of Theorem 2 in \cite{Ondrejat:2005}, let $(x_n^* : n \in \mathbb{N})$ be a dense subset in the unit sphere of $X^*$, and $G_X$ be the Hilbert space given by the completion of $X$ using the norm $\| x \|^2_{G_X} = \sum_{n=1}^{\infty} |\inner{x}{x_n^*}|^2$.

The inclusion $j : X \to G_X$ is continuous, and so we can define an orthogonal C.M.V.M. with continuous paths given by 
$$N(t, A)(g) = M(t, A)(j^* g), \qquad \forall \, g \in G_X.$$
It is easy to verify that this process has operator-valued quadratic variation given by 
\begin{equation}\label{eqOperatorQuaVariMRepresTheoremContinuousBanach}
\Gamma_{N}([0,t] \times A) = \int_0^t \int_A\; j \Phi (r,u) Q\Phi^*(r,u) j^*\; \lambda(du) dr. 
\end{equation}
   Observe that $j \Phi (r,u) Q\Phi^*(r,u) j^*\in \mathcal{K}^+(G)$, so we can apply the result in Step 1, obtaining that there is a cylindrical $Q$-white noise measure in $H$ based on $\left(\mathbb{R}_{+} \times U, \mathcal{B}\left(\mathbb{R}_{+}\right) \otimes \mathcal{B}(U),Leb \otimes \lambda\right)$  defined with respect to an extended probability space $(\Omega \times \widetilde{\Omega}, \mathcal{F} \times \widetilde{\mathcal{F}}, \Prob \times \widetilde{\Prob})$ adapted to the filtration $\{\mathcal{F}_{t} \times \widetilde{\mathcal{F}}_{t}\}$, such that 
\begin{eqnarray*}
M(t,A) (j^*g)=N(t,A)(\omega)(g)&=& \int_0^t\int_A j\Phi(\omega,r,u)W(dr,du)(\omega, \widetilde{\omega})(g)\\
&=&\int_0^t \int_A \Phi(\omega,r,u)W(dr,du)(\omega, \widetilde{\omega})(j^*g).
\end{eqnarray*}
Since $j$ is injective, then $\overline{j^*(G_X)}=X^*$ we conclude that
$$
M(t,A)(x^*)=\int_0^t \int_A \Phi(r,u)W(dr,du) \,(x^*), \quad \forall x^* \in X^*,
$$
which finalizes the proof.
\end{proof}

\section{Special cases for the representation theorem}\label{sectSpecialMartingaleRepre}

\subsection{Hilbert-space valued martingale representation}

If we assume that $M=(M(t, A): t \geq 0, A \in \calA)$ is $G$-valued for a separable Hilbert space $G$, we can  strengthen  the conclusion of Theorem \ref{martingalerepresentationtheoremContinuousCase} to obtain a classical version of the martingale representation theorem. Let $(g_{n}: n\in \N)$ be an orthonormal basis in $G$. Because $M$ is square-integrable, for any given $t >0$ and $A \in \calA$ by \eqref{eqOperatorQuaVariMRepresTheoremContinuous}  we have 
\begin{eqnarray*}
\Exp \left[ \norm{M(t,A)}_{G}^{2} \right]   & = & \sum_{n=1}^{\infty} \Exp \left[ 
\quadraVari{ \inner{M(t,A)}{g_{n}}_{G}}  \right] \\
& = & \sum_{n=1}^{\infty} \Exp
\int_0^t \int_A\; \inner{\Phi (r,u) Q\Phi^*(r,u)g_{n}}{g_{n}}_{G} \; \lambda(du) dr \\
& = &  \Exp
\int_0^t \int_A\; \mbox{tr} \left( \Phi (r,u) Q\Phi^*(r,u) \right) \lambda(du) dr. 
\end{eqnarray*}
From the above and a standard monotone class argument, we conclude that for $\Prob \otimes dt\otimes \lambda$-a.e. $(\omega,t,u)$ we have 
\begin{equation}\label{eqOperaQuadVariaIsHilbSchmitInHValuedMVM}
 \Phi (\omega, t,u) Q\Phi^*(\omega, t,u) \in \mathcal{L}^+_{1}(H) \Leftrightarrow \Phi (\omega, t,u) Q^{1/2} \in \mathcal{HS}(H,G) \Leftrightarrow \Phi (\omega, t,u) \in \mathcal{HS}(H_{Q},G).    
\end{equation}   

We are ready to prove the following version of the martingale representation theorem for Hilbert-valued orthogonal martingale-valued measures.

\begin{theorem}
[$G$-valued martingale representation]
\label{theoremClassicalMartingaleRepresentation}
Let $M=(M(t, A): t \geq 0, A \in \calA)$ be a $G$-valued orthogonal M.V.M. with continuous paths and which possesses a unique quadratic variation.  
Assume that the operator-valued quadratic variation of $M$ is given by 
\begin{equation}\label{eqOperatorQuaVariMRepresTheoremContinuousHilbertvalued}
 \Gamma_{M}([0,t] \times A) = \int_0^t \int_A\; \Phi (r,u) Q\Phi^*(r,u)\; \lambda(du) dr, 
\end{equation}
where $Q$ is a bounded, non-negative, self-adjoint operator on $H$ with  representation  \eqref{representationCovariaQDefiCylinWhineNoiseMeasure}, $\lambda$ is a $\sigma$-finite measure on $U$, which is finite on the sets in $\mathcal{A}$, and $\Phi: \Omega \times[0,T] \times U \rightarrow \mathcal{L}(H,G)$ is such that $ (\omega,t,u) \mapsto \Phi(\omega,t,u)h$ 
is $\mathcal{P}_{T}\otimes \mathcal{B}(U)/\mathcal{B}(G)$-measurable for every $h \in H$.
Then $\Phi (\omega, t,u) \in \mathcal{HS}(H_{Q},G)$ and there exists a 
cylindrical $Q$-white noise measure $W=(W(t, A): t \geq 0, A \in \calA)$ on $H$ based on $\left(\mathbb{R}_{+} \times U, \mathcal{B}\left(\mathbb{R}_{+}\right) \otimes \mathcal{B}(U),Leb \otimes \lambda\right)$
in $H$ based on $\left(\mathbb{R}_{+} \times U, \mathcal{B}\left(\mathbb{R}_{+}\right) \otimes \mathcal{B}(U),Leb \otimes \lambda\right)$  defined with respect to an extended probability space $(\Omega \times \widetilde{\Omega}, \mathcal{F} \times \widetilde{\mathcal{F}}, \Prob \times \widetilde{\Prob})$ adapted to the filtration $\{\mathcal{F}_{t} \times \widetilde{\mathcal{F}}_{t}\}$, such that 
 $$
M(t,A)(\omega)= \int_0^t\int_A \Phi(\omega,r,u)W(dr,du)(\omega, \widetilde{\omega}),
$$ 
where the last equality must be interpreted as $G$-valued random variables.
\end{theorem} 
\begin{proof}
First, by \eqref{eqOperaQuadVariaIsHilbSchmitInHValuedMVM} we have that 
$\Psi(\omega,t,u) \Psi^{*}(\omega,t,u) \in \mathcal{K}^{+}(H)$
for $\Psi(\omega,t,u)=\Phi(\omega,t,u) Q^{1/2}$. Hence all the hypothesis of Theorem \ref{martingalerepresentationtheoremContinuousCase} are satisfied. Therefore there exists a 
cylindrical $Q$-white noise measure in $H$ based on $\left(\mathbb{R}_{+} \times U, \mathcal{B}\left(\mathbb{R}_{+}\right) \otimes \mathcal{B}(U),Leb \otimes \lambda\right)$  defined with respect to an extended probability space $(\Omega \times \widetilde{\Omega}, \mathcal{F} \times \widetilde{\mathcal{F}}, \Prob \times \widetilde{\Prob})$ adapted to the filtration $\{\mathcal{F}_{t} \times \widetilde{\mathcal{F}}_{t}\}$, such that 
 $$
\inner{M(t,A)(\omega)}{g}_G= \int_0^t\int_A \Phi(\omega,r,u)W(dr,du)(\omega, \widetilde{\omega})(g), \quad \forall \, g \in G. 
$$ 
By \eqref{eqOperaQuadVariaIsHilbSchmitInHValuedMVM}, we can assume $\Phi (
\omega,t,u) \in \mathcal{HS}(H_{Q},G)$ then by Proposition \ref{Propcompatibilityofintegrals} we can conclude that the integral is radonified and so 
by \eqref{eqCompatibilityOfIntegrals} 
$$
\inner{M(t,A)(\omega)}{g}_G= \inner{\int_0^t\int_A \Phi(\omega,r,u)W(dr,du)(\omega, \widetilde{\omega})}{g}_G, \quad \forall \, g \in G, 
$$ 
and we obtain the desired result. 
\end{proof}

%\subsection{Example: cylindrical $Q$-Wiener process and cylindrical $Q$-white noise measure}

\subsection{Representation theorem for a cylindrical square integrable martingale}

In \cite{Ondrejat:2005} the author proves a martingale representation theorem for cylindrical continuous square integrable martingales.
In this section we apply our martingale representation theorem for cylindrical martingale-valued measures to obtain an alternative proof for the result in \cite{Ondrejat:2005}.

Our method is based in the choice of a particular form for the Q-cylindrical White noise measure which is introduced in the following example.

\begin{example}\label{examQWhiteNoiseCylindricaWiener}
Let $(b^n)_{n \in \N}$ be a sequence of independent standard Brownian motions. Consider an one point set $U=\{a\}$ y $\mathcal{A}=\mathcal{P}(U)$. Note that for every $n \in \N$, $W_n(t,A)=b_t^n \delta_a(A)$ define a white noise measure on $\left(\mathbb{R}_{+} \times \{a\}, \mathcal{B}\left(\mathbb{R}_{+}\right) \otimes \mathcal{P}(\{a\}),Leb \otimes \delta_a\right)$ and the sequence $(W_n)_{n \in \N}$ is independent. Now, we can define a cylindrical $Q$-white noise measure by
\begin{eqnarray*}
W(t,A) h&=& \sum_{n=1}^\infty \lambda_n^{1/2} W_n(t,A) (h_n,h)_H\\
&=&\sum_{n=1}^\infty \lambda_n^{1/2} b_t^n \delta_a(A) (h_n,h)_H\\
&=&B_t(h) \delta_a(A),
\end{eqnarray*}
where $B_t(h)=\displaystyle \sum_{n=1}^\infty \lambda_n^{1/2} b_t^n (h_n,h)_H$ defines a cylindrical $Q$-Wiener process 
if $Q$ is a trace class operator. 
\end{example}

For the reciprocal of the above example, we gives the following result. 

\begin{proposition}
\label{propcylindricalWienerprocess}
Let $(W(t,A), t \geq 0, \, A \in \mathcal{A})$ be a cylindrical $Q$-white noise measure on $\left(\mathbb{R}_{+} \times \{a\}, \mathcal{B}\left(\mathbb{R}_{+}\right) \otimes \mathcal{P}(\{a\}),Leb \otimes \delta_a\right)$. Then there exists a cylindrical Wiener process $B=(B_t: t \geq 0)$ in $H$ with covariance $Q$ such that
$$W(t,A)=B_t \delta_a(A), \quad \forall A \in \mathcal{P}(\{a\}).$$
\end{proposition}

\begin{proof}
In view of \eqref{eqDefWN} and of Example \ref{examQWhiteNoiseCylindricaWiener} it suffices to show that for every $n \in \N$ we have $W_n(t,A)=b_t^n \delta_a(A)$ where $(b^n: n \in \N)$  is a sequence of real-valued independent standard Brownian motions.  But this assertion is indeed a direct consequence of 
 Proposition II-1 of \cite{KarouiMeleard:1990}. 
\end{proof}

\begin{corollary}
\label{coroHvaluedWN}
Let $(W(t,A), t \geq 0, \, A \in \mathcal{A})$ be an $H$-valued $Q$-white noise measure on $\left(\mathbb{R}_{+} \times \{a\}, \mathcal{B}\left(\mathbb{R}_{+}\right) \otimes \mathcal{P}(\{a\}),Leb \otimes \delta_a\right)$, where $Q$ is a trace class operator. Then there exists an $H$-valued $Q$-Wiener process $B=(B_t: t \geq 0)$ in $H$ such that
$$W(t,A)=B_t \delta_a(A), \quad \forall A \in \mathcal{P}(\{a\}).$$
\end{corollary}
\begin{proof}
This is a consequence of the result in Proposition \ref{propcylindricalWienerprocess} and Remark \ref{remarkQtraceclass}. 
\end{proof}

The main result of this section is the following. 

\begin{theorem}
\label{theoreprescylindricalmartingale}
Let $X$ be a separable, reflexive Banach space and let $(m_t: t \geq 0)$ be a cylindrical square integrable martingale on $X^*$ with operator quadratic covariation satisfying
$$\langle m(x^*_1), m(x^*_2)\rangle_t=\int_0^t \Bigl(\Phi (r)Q \Phi(r)^* x^*_1, x^*_2\Bigr)_H\,dr, \quad \forall x^*_1, x^*_2 \in X^*,$$ 
where $Q$ is a bounded, non-negative, self-adjoint operator on $H$ with  representation  \eqref{representationCovariaQDefiCylinWhineNoiseMeasure} and $\Phi: \Omega \times[0,T]  \rightarrow \mathcal{L}(H,X)$ is $ (\omega,t) \mapsto \Phi(\omega,t)h$ 
is $\mathcal{P}_{T}/\mathcal{B}(X)$-measurable for every $h \in H$ and such that for $\Psi(\omega,t)=\Phi(\omega,t) Q^{1/2}$, we have $\Psi(\omega,t) \Psi^{*}(\omega,t) \in \mathcal{K}^{+}(X)$. Then there exists a cylindrical Wiener process $B=(B_t: t \geq 0)$ in $H$ with covariance $Q$ defined on an extended probability space $(\Omega \times \widetilde{\Omega}, \mathcal{F} \times \widetilde{\mathcal{F}}, \Prob \times \widetilde{\Prob})$ adapted to the filtration $\{\mathcal{F}_{t} \times \widetilde{\mathcal{F}}_{t}\}$ such that
\begin{equation}
\label{equreprescylindricalmartingale}
m_t(\omega)=\int_0^t \Phi(\omega, r)\, dB_r (\omega, \widetilde{\omega}).
\end{equation}
\end{theorem}

\begin{proof} Let $M(t,A)=m_t \delta_a(A)$, where we consider an one point set $U=\{a\}$ and $\mathcal{A}=\mathcal{P}(U)$. 
In that case, the operator-valued quadratic variation of $M$ is given by
$$
\Gamma_{M}([0,t] \times A) = \delta_a(A)\int_0^t \; \Phi (r,u) Q\Phi^*(r,u)\; dr.
$$
Hence, using Theorem \ref{martingalerepresentationtheoremContinuousCase} we can conclude that 
there exists a 
cylindrical $Q$-white noise measure in $H$ based on $\left(\mathbb{R}_{+} \times \{a\}, \mathcal{B}\left(\mathbb{R}_{+}\right) \otimes \mathcal{P}(\{a\}),Leb \otimes \delta_a\right)$ defined with respect to an extended probability space $(\Omega \times \widetilde{\Omega}, \mathcal{F} \times \widetilde{\mathcal{F}}, \Prob \times \widetilde{\Prob})$ adapted to the filtration $\{\mathcal{F}_{t} \times \widetilde{\mathcal{F}}_{t}\}$, such that 
 $$
M(t,A)(\omega)= \int_0^t\int_A \Phi(\omega,r)W(dr,du)(\omega, \widetilde{\omega}).
$$
Now, with the help of Proposition \ref{propcylindricalWienerprocess}, we can conclude that
there exists a cylindrical $Q$-Wiener process $B=(B_t: t \geq 0)$ such that $W(t,A)=B_t \delta_a$. Then for each $t \geq 0$ we have 
\begin{equation}
\label{theocylindricalmartingalerepres}
m_t (\omega)\, \delta_a(A)=M(t,A)(\omega)=\left(\int_0^t \Phi(r)\, dB_r (\omega, \widetilde{\omega})\right)  \, \delta_a (A).
\end{equation}
Now if we take $A=\{a\}$ in (\ref{theocylindricalmartingalerepres}) we obtain (\ref{equreprescylindricalmartingale}).

\end{proof}

In the case where we have a (classical) square integrable martingale taking values in a Banach space, and with the assumption that $Q$ is of trace class, we can strengthen the conclusions of Theorem \ref{theoreprescylindricalmartingale} to obtain a representation as a stochastic integral with respect to a Hilbert-space valued $Q$-Wiener process. 

\begin{theorem}
\label{theoRepClassMartBan}
Let $X$ be a separable, reflexive Banach space and let $(m_t: t \geq 0)$ be an $X$-valued square integrable martingale, with operator quadratic covariation satisfying
$$\langle x^*_1(m), x^*_2(m)\rangle_t=\int_0^t \Bigl(\Phi (r)Q \Phi(r)^* x^*_1, x^*_2\Bigr)_H\,dr, \quad \forall x^*_1, x^*_2 \in X^*,$$ 
for $Q\in \mathcal{K}^+(H)$ and $\Phi: \Omega \times[0,T]  \rightarrow \mathcal{L}(H,X)$ is $ (\omega,t) \mapsto \Phi(\omega,t)h$ 
is $\mathcal{P}_{T}/\mathcal{B}(X)$-measurable for every $h \in H$. Then there exists an $H$-valued $Q$-Wiener process $B=(B_t: t \geq 0)$, defined on an extended probability space $(\Omega \times \widetilde{\Omega}, \mathcal{F} \times \widetilde{\mathcal{F}}, \Prob \times \widetilde{\Prob})$ adapted to the filtration $\{\mathcal{F}_{t} \times \widetilde{\mathcal{F}}_{t}\}$ such that
\begin{equation}
\label{equreprescylindricalmartingaleclassic}
m_t(\omega)=\int_0^t \Phi(\omega, r)\, dB_r (\omega, \widetilde{\omega}).
\end{equation}
This equality must be taken as $X$-valued random variables.
\end{theorem}

\begin{proof}
    Notice that $\Psi(\omega,t)=\Phi(\omega,t) Q^{1/2}$ satisfies $\Psi(\omega,t) \Psi^{*}(\omega,t) \in \mathcal{K}^{+}(X)$. Then we proceed as in the proof of Theorem \ref{theoreprescylindricalmartingale}, but instead of using Proposition \ref{propcylindricalWienerprocess} we apply Corollary \ref{coroHvaluedWN}.
\end{proof}

\begin{remark}
A martingale representation theorem for Hilbert space-valued square integrable martingale can be found in the literature, for example in (\cite{DaPratoZabczyk}, Theorem 8.2). 
Our Theorem \ref{theoRepClassMartBan} extends these results in two directions. First, we allow our square integrable martingale to take values in a separable reflexive Banach space. Second, we only require  $\Phi(\omega,t) \in \mathcal{L}(H,X)$ while the authors in  \cite{DaPratoZabczyk} require (for $X$ Hilbert) that $\Phi(\omega,t) \in \mathcal{HS}(H_{Q},X)$. The later is a more demanding condition in $\Phi$ than ours.   
\end{remark}

\section{Application: Characterization of solutions to the weak martingale problem}\label{sectWeakMartingaleProblem}

In this section we apply Theorem \ref{martingalerepresentationtheoremContinuousCase} to characterize solutions of a weak martingale problem as weak solutions to a particular stochastic differential equation (SDE) driven by cylindrical white noise measure noise. Given a probability space $(\Omega,\calF,\Prob)$, we study the following  SDE in differential form:
% \[
% dX_t = \psi(t)\, dA(t) + \int_U \Phi(t,u)\, M(dt,du).
% \]
% with initial condition $X_{0}=\xi$, where 

% \[
% X_t = \xi + \int_0^t \psi(s)\, dA(s) + \int_0^t \int_U \Phi(s,u)\, M(ds,du).
% \]
%In the standard SDE form:
\begin{equation}
\label{eqSDE}
   dX_t = a(t,X_t)\,dt + \int_U b(t,u,X_t)\, W(dt,du),
\end{equation}
where
\begin{enumerate}
\item $a:[0,T]\times G \to G$ and $b:[0,T]\times U \times G \to \calL(H,G)$ satisfy that for each $x\in C([0,T],G)$ the mapping $
t\mapsto a(t,x_t)$ is  $\mathcal{B}([0,T])/\mathcal{B}(G)$ measurable and $  (t,u) \mapsto b(t,u,x_t)h$ is $ \mathcal{B}([0,T]) \otimes \mathcal{B}(U)/\mathcal{B}(G)$-measurable for every $h \in H$.  
\item $W$ is a cylindrical $Q$-white noise measure based on $\left(\mathbb{R}_{+} \times U, \mathcal{B}\left(\mathbb{R}_{+}\right) \otimes \mathcal{B}(U),Leb \otimes \lambda\right)$, where $Q$ is a bounded, non-negative, self-adjoint operator on $H$ with  representation  \eqref{representationCovariaQDefiCylinWhineNoiseMeasure} and $\|Q\|=1$, $\lambda$ is a $\sigma$-finite measure on $U$, which is finite on the sets in $\mathcal{A}$. 
\end{enumerate}

With the conditions above, we introduce the concepts of strong and weak distributional solutions to \eqref{eqSDE}.

\begin{definition}
\label{weakfunctsol}
A \emph{strong solution} to \eqref{eqSDE} on the probability space $(\Omega, \mathcal{F}, \Prob)$, with initial value $\xi$, is a $G$-valued continuous adapted process $X=(X_{t}: 0 \leq t \leq T)$ satisfying $\Prob$-a.e.
\begin{enumerate}
    \item  $\displaystyle \int_{0}^{T} \norm{a(t,X_{t})}_{G} dt< \infty.$ 

    \item
    $\displaystyle \int_{0}^{T} \int_{U} \norm{b(t,u,X_{t})Q^{1/2}}^2_{\mathcal{HS}(H,G)}\lambda(du)dt < \infty.$ 
    %\item $b(t,u,X_{t}(\omega))Q^{1/2} \in \mathcal{HS}(H,G)$ for each $(\omega,t,u)$ and   $\Prob$-a.e.  
    %$$ \int_{0}^{T} \int_{U} \norm{b(t,u,X_{t})Q^{1/2}}^{2}_{\mathcal{HS}(H,G)}\lambda(du)dt < \infty.  $$ 
    \item 
    $$\displaystyle  X_{t}= \xi + \int_{0}^{t} a(r,X_{r})dr + \int_{0}^{t} \int_{U} \, b(r,u,X_{r}) \, W(dr,du). 
    $$ 
    % \label{EQdefiweaksolution}  For every $g \in G$, $$\displaystyle \inner {X_{t}}{g}= \inner{\xi}{g} + \int_{0}^{t} \inner{a(r,X_{r})}{g}dr + \left(\int_{0}^{t} \int_{U} \, b(r,u,X_{r}) \, W(dr,du)\right)(g), 
    % $$ 
\end{enumerate}
\end{definition}

Observe that if $X$ is a strong  solution to \eqref{eqSDE}, then $X$ is a $G$-valued continuous semimartingale. In effect, since is defined $\omega$-wise as a Bochner integral, then $\int_{0}^{t} \, a(r,X_{r})dr$ is a $G$-valued continuous adapted process of finite variation. On the other hand, the integral $\displaystyle\int_{0}^{t} \int_{U} \, b(r,u,X_{r}) \, W(dr,du)$ is defined as a $G$-valued locally square integrable martingale using the theory of stochastic integration developed in \cite{CCFM:SPDE}.

\begin{definition}
A \emph{weak distributional solution} to \eqref{eqSDE} with initial distribution $\mu$ consists of a triplet $(\Omega, \mathcal{F}, \Prob)$, a  cylindrical $Q$-white noise measure based on $\left(\mathbb{R}_{+} \times U, \mathcal{B}\left(\mathbb{R}_{+}\right) \otimes \mathcal{B}(U),Leb \otimes \lambda\right)$ and a $G$-valued continuous adapted process $X=(X_{t}: 0 \leq t \leq T)$ for which $X_{0}$ has law $\mu$ and which is a strong solution to \eqref{eqSDE} on $(\Omega,\calF,\Prob)$.
\end{definition}

% We can think of solutions to equation \eqref{eqSDE} as those processes $X$ that coincide with the It\^{o} process corresponding to

% (see \cite[Definition 4.1]{CCFM:Ito}). $W$ is a cylindrical $(Q,\lambda)$-white noise measure

% For $f \in C^2(G;\R)$, if we use the It\^{o} formula in the continuous paths setting, we have
% \begin{align*}
% f(X_t) &= f(\xi) 
% + \int_0^t \int_U f'(X_s)\, b(s,u,X_s)\, M(ds,du) 
% + \int_0^t f'(X_s)\, a(s,X_s)\, ds \\
% &\quad + \tfrac{1}{2}\int_0^t \int_U 
% \operatorname{Tr}_{b(s,u,X_s)\, Q_M^{1/2}}\left( f''(X_s)\right) \,
% \operQuadraVari{M}(ds,du).
% \end{align*}
% (see for instance Corollary 4.9 in \cite{CCFM:Ito}).

% The infinitesimal generator is given by
% \[
% \int_0^t A_s f(X_s) ds = \int_0^tf'(X_s)\, a(s,X_s)ds + \tfrac{1}{2} \int_0^t \int_U 
% \operatorname{Tr}_{b(s,u,X_s)\, Q_M^{1/2}}\left( f''(X_s)\right) \operQuadraVari{M}(ds,du).
% \]

% The compensated martingale is
% \[
% M^f_t = f(X_t) - f(\xi) - \int_0^t A_s f(X_s)\, ds.
% \]

% The quadratic variation of $X_t$ is
% \[
% [X]_t = \Bigg[ \int_0^t \int_U b(s,u,X_s)\, M(ds,du)\Bigg]_t
% = \int_0^t \int_U b(s,u,X_s)\, Q_M\, b(s,u,X_s)^* \, 
% \langle\!\langle M \rangle\!\rangle(ds,du).
% \]

Now we define the martingale problem associated to the stochastic differential equation \eqref{eqSDE}. Before, we need to establish some notation. 

Denote by $f\in C_{u,loc}^{2}(G)$ the collection of all the functions $f:G \rightarrow G$ which are twice differentiable and for which $f_{xx}$ is uniformly continuous on bounded subsets of $G$. 

For a continuous bilinear form $\zeta: G\times G \rightarrow G$ and $S,R \in \mathcal{HS}(H,G)$ we use $\mbox{Tr}_{S,R}(\zeta)$  to denote the \emph{trace} defined by
$$ \mbox{Tr}_{S,R}(\zeta)=\sum_{j=1}^{\infty} \zeta(S h_{j},R h_{j}), $$

where $(h_{j})_{j \in \N}$ is an orthonormal basis in $H$.

\begin{definition}
Let $P$ be a probability distribution on $C([0,T],G)$. We say that $P$ solves the \emph{weak martingale problem} for $(a,b,Q,\lambda)$ if for any $G$-valued continuous adapted process $X$ defined on a probability space $(\Omega,\mathcal{F},\mathbb{P})$ whose probability distribution is $P$, we have for  
every $f\in C_{u,loc}^{2}(G)$,
\[
M^f_t = f(X_t) - f(\xi) - \int_0^t A_s f(X_s)\, ds
\]
is a $G$-valued continuous martingale on $(\Omega,\mathcal{F},\mathbb{P})$, where the infinitesimal generator $A_t$ satisfies
$$
\int_0^t A_s f(x_s)\, ds = \tfrac{1}{2} \int_0^t \int_U 
\operatorname{Tr}_{b(s,u,x_s)\, Q^{1/2}}\left( f''(x_s)\right) \lambda(du) ds
+ \int_0^t f'(x_s)\, a(s,x_s)\, ds,
$$
for any $x\in C([0,T],G)$. In such a case, we also say that $X$ solves the martingale problem. 
\end{definition}

%  A $G$-valued continuous process $X$ on a probability space $(\Omega,\mathcal{F},\mathbb{P})$ solves the \emph {martingale problem} for $(a,b,Q,\lambda)$ if for every $f\in C^\infty(G)$,
% \[
% M^f_t = f(X_t) - f(\xi) - \int_0^t A_s f(X_s)\, ds
% \]
% is a martingale in $(\Omega,\mathcal{F},\mathbb{P})$, where the infinitesimal generator $A_t$ satisfies
% $$
% \int_0^t A_s f(x_s)\, ds = \tfrac{1}{2} \int_0^t \int_U 
% \operatorname{Tr}_{b(s,u,x_s)\, Q_M^{1/2}}\left( f''(x_s)\right) \lambda(du) ds
% + \int_0^t f'(x_s)\, a(s,x_s)\, ds,
% $$
% for any $x\in C([0,T],G)$.

% \begin{theorem}
% $X$ solves the martingale problem for $(a,b,Q,\lambda)$ if and only if $X$ is a weak solution to the SDE
% \begin{equation}
% \label{SDEequation}
% dX_t = a(t,X_t)\,dt + \int_U b(t,u,X_t)\, W(dt,du),
% \end{equation}
% \end{theorem}

The following is the main result of this section. It establish a one to one relationship between the existence of a weak distributional solution and the existence of a solution to the martingale problem. 

\begin{theorem}\label{theoWeakSolutionEquival}
Given $(a,b,Q,\lambda)$ and a probability distribution $P$ on $C([0,T],G)$ the following statements are equivalent:
\begin{enumerate}
\item  Equation \eqref{eqSDE} for $(a,b,Q,\lambda)$  has a weak distributional solution with distribution $P$. 
\item $P$ solves the weak martingale problem for $(a,b,Q,\lambda)$. 
\end{enumerate}
\end{theorem}
\begin{proof} 
Assume a process $X$ with distribution $P$ and a cylindrical $Q$-Wiener white noise $W$ defined both on a probability space $(\Omega, \mathcal{F}, \Prob)$ are a weak distributional solution to \eqref{eqSDE}. By an application of the It\^{o} formula in the continuous paths setting (see for instance Corollary 4.9 in \cite{CCFM:Ito}), one has that
\begin{align*}
M_t^f&=f(X_t) - f(\xi)-\int_0^t A_s f(X_s)\,ds = \int_0^t \int_U f'(X_s)\, b(s,u,X_s)\, W(ds,du), 
\end{align*}
which is a martingale, for any $f\in C_{u,loc}^{2}(G)$.

Conversely, assume that $X$ defined on $(\Omega, \mathcal{F}, \Prob)$ solves the martingale problem for $(a,b,Q,\lambda)$. If we take $f(x)=x$, one can deduce that
\begin{equation}
\label{martingaleproblem1}
M_t = X_t - X_0 - \int_0^t a(s,X_s)\,ds
\end{equation}
is a $G$-valued continuous martingale. Note that $\langle M \rangle_t=\langle X \rangle_t$. 

Using integration by parts for $G$-valued semimartingales (see for instance Theorem 26.5 of \cite{Metivier}), we have that
$$
\|X_t\|_G^2-\|X_0\|^2_G=2 \int_0^t \widetilde{X}_s \, dX_s+\langle X \rangle_t =2 \int_0^t \widetilde{X}_s \, dX_s+\langle M \rangle_t,
$$
where $\widetilde{X}_s(g)=(X_s,g)_G$ for every $g \in G$. 
Defining $\displaystyle \int_0^t X_s\, dM_s=\int_0^t \widetilde{X}_s\,dM_s$ and with the help of (\ref{martingaleproblem1}) one can deduce  $$\displaystyle \int_0^t X_s\, dM_s=\int_0^t \widetilde{X}_s\, dX_s-\int_0^t (X_s, a(s,X_s))_G\, dx.$$ 
We conclude that
\begin{equation}
\label{integrationbypartsformula}
\|X_t\|_G^2-\|X_0\|_G^2=2 \int_0^t X_s\, dM_s+2\int_0^t (X_s, a(s, X_s))_G\,ds+\langle M \rangle_t.
\end{equation}
On the other hand, taking $f(x)=(x,x)_G z$ for $z \in G$ with $\|z\|_G=1$, since $X$ satisfies the martingale problem we can infer that
\begin{equation}
\label{martingaleproblem2}
\widetilde{M}_t=\|X_t\|_G^2 - \|X_0\|_G^2 - \int_0^t 2\langle X_s,a(s,X_s)\rangle_G ds - \int_0^t \int_{U} \norm{b(s,u,X_s)Q^{1/2}}_{\mathcal{HS}(H,G)}^2 \lambda(du)ds
\end{equation}
is also a martingale. Using (\ref{integrationbypartsformula}) and (\ref{martingaleproblem2}), one has that
$$
\widetilde{M}_t=2 \int_0^t X_s\, dM_s+\langle M \rangle_t-\int_0^t \int_{U}\norm{b(s,u,X_s)Q^{1/2}}_{\mathcal{HS}(H,G)}^2 \, \lambda(du)ds.
$$
Hence 
$$\displaystyle \langle M \rangle_t-\int_0^t \int_{U}\norm{b(s,u,X_s)Q^{1/2}}_{\mathcal{HS}(H,G)}^2\,\lambda(du)ds$$ 
is a martingale, which has a locally finite variation. As a result 
$$\displaystyle \langle M \rangle_t=\int_0^t \int_{U} \norm{b(s,u,X_s)Q^{1/2}}^2_{\mathcal{HS}(H,G)}\, \lambda(du)ds,$$ 
in particular we have 
$$ \langle M(g_1), M(g_2)\rangle_t=\int_0^t \int_{U}  \Bigl(b(s,u,X_s)Q b(s,u,X_s)^* g_1, g_2\Bigr)_G \,\lambda(du) ds, \quad \forall g_1, g_2 \in G.$$
Applying Theorem \ref{theoremClassicalMartingaleRepresentation} (here one can take the $G$-valued orthogonal martingale-valued measure $\widetilde{M}(t,A)=M_t \delta_a(A)$, where we consider an one point set $U=\{a\}$ and $\mathcal{A}=\mathcal{P}(U)$) we have $b(s,u,X_s) \in \mathcal{HS}(H,G)$ and there exists a 
cylindrical $Q$-white noise measure on $H$ based on $\left(\mathbb{R}_{+} \times U, \mathcal{B}\left(\mathbb{R}_{+}\right) \otimes \mathcal{B}(U),Leb \otimes \lambda\right)$  defined with respect to an extended probability space $(\Omega \times \widetilde{\Omega}, \mathcal{F} \times \widetilde{\mathcal{F}}, \Prob \times \widetilde{\Prob})$ adapted to the filtration $\{\mathcal{F}_{t} \times \widetilde{\mathcal{F}}_{t}\}$, such that 
 \begin{equation*}
\label{representationformulaforb}
M_t= \int_0^t\int_U b(s,u,X_s)W(ds,du)(\omega, \widetilde{\omega}).
\end{equation*}
Then with the help of (\ref{martingaleproblem1}) we conclude that $X_t$ satisfies \eqref{eqSDE}. Thus, $X$ is a weak distributional solution to \eqref{eqSDE} with distribution $P$. 
\end{proof}

\section{Appendix}

The following proposition is a fundamental step of the proof of the martingale representation theorem (Theorem \ref{martingalerepresentationtheoremContinuousCase}). Moreover, it extends Lemma 8.9 in \cite{PeszatZabczykSPDE}.

\begin{proposition}\label{propCanonicalDecompoRandomCompact}
Let $\Phi: \Omega \times [0,T] \times U \rightarrow \mathcal{K}^{+}(H)$ be such that for every $h_{1}, h_{2} \in H$ the mapping $(\omega, t, u) \mapsto \inner{\Phi(\omega, t, u)h_{1}}{h_{2}}$ is $\mathcal{P}_{T} \otimes \mathcal{B}(U)$-measurable. Then, there exists a decreasing sequence  $(\lambda_{n}:n \in \N)$ of non-negative $\mathcal{P}_{T} \otimes \mathcal{B}(U)$-measurable processes with $\lambda_{n} \rightarrow 0$, and a sequence $(\varphi_{n}: n \in \N)$ of $H$-valued $\mathcal{P}_{T} \otimes \mathcal{B}(U)$-measurable processes such that 
\begin{equation}\label{eqCanonicalDecompoRandomCompact}
  \Phi(\omega, t, u) = \sum_{n=1}^{\infty} \lambda_{n}(\omega, t, u) (\varphi_{n} (\omega, t, u)  \otimes \varphi_{n} (\omega, t, u)), \quad \forall \,  \omega \in \Omega, t \geq 0, u \in U.  
\end{equation}
Moreover, we can choose $\lambda_{n}$ and $\varphi_{n}$ such that for all $\omega \in \Omega$, $t \geq 0$, $u \in U$,
\begin{equation}\label{eqNormFunctionsG}
\norm{\varphi_{n}(\omega, t, u)}_{H}=
\begin{cases}
1 & \mbox{ if } \lambda_{n}(\omega, t, u) >0, \\
0 & \mbox{ if } \lambda_{n}(\omega, t, u) =0,
\end{cases}    
\end{equation}
and 
\begin{equation}\label{eqOrthogonalFunctionsG}
\inner{\varphi_{m} (\omega, t, u) }{ \varphi_{n} (\omega, t, u)}_{H}=\delta_{m,n},  \quad \mbox{ for } m \neq n. 
\end{equation}    
\end{proposition}

For our proof of Proposition \ref{propCanonicalDecompoRandomCompact} we will need the following theorem. 

\begin{theorem}[Kuratowski–Ryll-Nardzewski]\label{theoremKuratowskiRyllNardzewski}
Let $E$ be a compact metric space and let $\psi: E \times \widetilde{\Omega} \rightarrow \R$ be a mapping such that $\psi(x, \cdot)$ is measurable for arbitrary $x \in E$ and $\psi(\cdot, \widetilde{\omega})$ is a continuous mapping for arbitrary $\widetilde{\omega} \in \widetilde{\Omega}$. Then there exist an $E$-valued measurable function $X:\widetilde{\Omega} \rightarrow E $ such that 
$$ \psi(X(\widetilde{\omega}),\widetilde{\omega})=\sup_{x \in E} \psi(x,\widetilde{\omega}), \quad \widetilde{\omega} \in \widetilde{\Omega}.$$
\end{theorem}

\begin{proof}[Proof of Proposition \ref{propCanonicalDecompoRandomCompact}] We apply Theorem \ref{theoremKuratowskiRyllNardzewski} to $\widetilde{\Omega} = \Omega \times [0,T] \times U$ equipped with the $\sigma$-algebra  $\widetilde{\mathcal{F}}=\mathcal{P}_{T} \otimes \mathcal{B}(U)$, the set $E=\{h \in H: \norm{h} \leq 1  \}$ endowed with the weak topology and the function 
$$ \psi(h,\omega,t,u)=\inner{\Phi(\omega,t,u)h}{h},$$
for $h \in E$, $(\omega,t,u) \in \widetilde{\Omega}$. All conditions of  Theorem \ref{theoremKuratowskiRyllNardzewski} are satisfied, so there exists a $H$-valued $\mathcal{P}_{T} \otimes \mathcal{B}(U)/\mathcal{B}(H)$-measurable process $g_{1}$ such that 
\begin{equation*}
\inner{\Phi(t,\omega,u)\varphi_{1}(\omega,t,u)}{\varphi_{1}(\omega,t,u)}_{H} 
= \sup_{\norm{h} \leq 1} \inner{\Phi(\omega,t,u)h}{h}_{H} = \lambda_{1}(\omega,t,u).
\end{equation*}
By the properties of $\Phi$ we have $\lambda_{1}$ is   $\mathcal{P}_{T} \otimes \mathcal{B}(U)$-measurable, non-negative, and indeed  $\lambda_{1}(\omega,t,u) = \norm{\Phi(\omega,t,u)}_{\mathcal{L}(H)}$. Observe that if  $\lambda_{1}(\omega,t,u)>0$ we can choose $\varphi_{1}(\omega,t,u)$ such that $\norm{\varphi_{1}(\omega, t, u)}_{H}=1$, otherwise we can choose $\varphi_{1}(\omega, t, u)=0$. Moreover, being $\Phi(t,\omega,u)$ self-adjoint and by the definition of $\varphi_{1}(\omega,t,u)$  we have 
$$ \Phi(t,\omega,u)\varphi_{1}(\omega,t,u)=\lambda_{1}(\omega,t,u)  \varphi_{1}(\omega,t,u). $$

We can construct inductively the sequences $\{\lambda_{n}\}$ and $\{\varphi_{n}\}$. In effect, assume we have $\lambda_{1}, \dots, \lambda_{n}$ and $\varphi_{1}, \dots, \varphi_{n}$ with the properties in the statement of the theorem. Let 
$$ \Phi_{n}(\omega,t,u) = \Phi(\omega,t,u)-\sum_{j=1}^{n} \lambda_{j}(\omega,t,u) \varphi_{j}(\omega,t,u) \otimes \varphi_{j}(\omega,t,u). $$
We apply  Theorem \ref{theoremKuratowskiRyllNardzewski} to 
$$ \psi_{n}(h,\omega,t,u)=\inner{\Phi_{n}(\omega,t,u)h}{h}_{H},$$
with $\widetilde{\Omega}$,   $\widetilde{\mathcal{F}}$ and $E$ as above, showing the existence of $\varphi_{n+1}$  such that  
\begin{equation*}
\inner{\Phi_{n}(t,\omega,u)\varphi_{n+1}(\omega,t,u)}{\varphi_{n+1}(\omega,t,u)}_{H} 
= \sup_{\norm{h} \leq 1} \inner{\Phi_{n}(\omega,t,u)h}{h}_{H} = \lambda_{n+1}(\omega,t,u).
\end{equation*}
As before,  if  $\lambda_{n+1}(\omega,t,u)>0$ we can choose $\varphi_{n+1}(\omega,t,u)$ such that $\norm{\varphi_{n+1}(\omega, t, u)}_{H}=1$, otherwise we can choose $\varphi_{n+1}(\omega, t, u)=0$.

Furthermore, if  $\lambda_{n+1}(\omega,t,u)>0$ we must have  $\inner{\varphi_{m} (\omega, t, u) }{ \varphi_{n+1} (\omega, t, u)}_{H}=0$ for $1 \leq m \leq n$.
To see why this is true, observe that as part of our induction hypothesis we are assuming 
$$ \Phi(t,\omega,u)\varphi_{m}(\omega,t,u)=\lambda_{m}(\omega,t,u)  \varphi_{m}(\omega,t,u), \mbox{ for } 1 \leq m \leq n, $$
then we have $\Phi_{n}(\omega,t,u) \varphi_{m}(\omega,t,u) = 0$ for all $1 \leq m \leq n$. This way 
$$\Phi_{n}(\omega,t,u) \left(\mbox{span}\{ \varphi_{m}(\omega,t,u): 1 \leq m \leq n \} \right)= \{ 0\}.$$
Hence, by the orthogonal decomposition theorem we conclude $\varphi_{n+1} (\omega, t, u) \in \mbox{span}\{ \varphi_{m}(\omega,t,u): 1 \leq m \leq n \}^{\perp}$. This completes the induction step.

Now, since $\lambda_{n+1}(\omega,t,u) = \norm{\Phi_{n}(\omega,t,u)}_{\mathcal{L}(H)}$ for $n \geq 0$ (with $\Phi_{0}(\omega,t,u)=\Phi(\omega,t,u)$), it is clear from the arguments in the above paragraphs that $\lambda_{n}(\omega,t,u) \geq \lambda_{n+1}(\omega,t,u)$. Hence the sequence $\{\lambda_{n} \}$ is decreasing. It remains to show that $\lambda_{n} \rightarrow 0$ and that \eqref{eqCanonicalDecompoRandomCompact} holds true.

First, by the orthogonality of the $\varphi_{n}$ we have for each $m,n \in \N$ that
$$ \norm{\Phi(t,\omega,u)\varphi_{m}(\omega,t,u)-\Phi(t,\omega,u)\varphi_{n}(\omega,t,u) }_{H}^{2}= \lambda_{m}(\omega,t,u)^{2} + \lambda_{n}(\omega,t,u)^{2}.  $$
Assume $\lambda_{n}(\omega,t,u)   \nrightarrow 0$. Then the exists $c>0$ such that $ \lambda_{n}(\omega,t,u) \geq c$ for all $n \in \N$. Therefore, 
$$ \norm{\Phi(t,\omega,u)\varphi_{m}(\omega,t,u)-\Phi(t,\omega,u)\varphi_{n}(\omega,t,u) }_{H}^{2} \geq 2c>0, \quad \forall m,n \in \N.$$
But the above is a contradiction because since $\Phi(t,\omega,u)\in \mathcal{K}^{+}(H)$ the sequence $(\Phi(t,\omega,u)\varphi_{n}(\omega,t,u): n \in \N ) $ is relatively compact, hence has a convergent subsequence contradicting the inequality above. Thus $\lambda_{n} \rightarrow 0$. 

Finally, since 
$$ \norm{\Phi(\omega,t,u)-\sum_{j=1}^{n} \lambda_{j}(\omega,t,u) \varphi_{j}(\omega,t,u) \otimes \varphi_{j}(\omega,t,u)}_{\mathcal{L}(H)}=\norm{ \Phi_{n}(\omega,t,u)}_{\mathcal{L}(H)}=\lambda_{n+1}(\omega,t,u), $$
we conclude \eqref{eqCanonicalDecompoRandomCompact} holds true. 
\end{proof}

% \section{Ideas}

% \begin{enumerate}
%     \item Hacer un ejemplo
%     \item Fubini d\'ebil
%     \item Teorema de representaci\'on de martingalas
%     \item Relación entre geometría del Banach y la radonificación de la integral
%     \item SDE's y explorar strong solutions. Martingale problems (weak solutions)
%     \item Explorar un Fubini tipo
%     $$
%     L^1(E,\rho,\mathcal{L}(Y^*,\mathcal{M}_T^2)) \equiv \mathcal{L}(Y^*,L^1(E,\rho,\mathcal{M}_T^2))
%     $$
% \end{enumerate} \bigskip

\noindent \textbf{Acknowledgments}  This work was partially supported by The University of Costa Rica through the grant ``C6163-Ecuaciones Diferenciales  Estoc\'{a}sticas en Espacios de Hilbert''. 

%The authors thank two anonymous referees for valuable comments and suggestions that contributed greatly to improve the presentation of this article. 

\smallskip

\noindent \textbf{Data Availability.} Data sharing not applicable to this article as no data sets were generated or analyzed during the current study.

\section*{Declarations}

\noindent \textbf{Conflict of interest} The authors have no conflicts of interest to declare that are relevant to the content of this article.


\begin{bibdiv}
    \begin{biblist}
\bib{AlvaradoFonseca:2021}{article}{
   author={Alvarado-Solano, A. E.},
   author={Fonseca-Mora, C. A.},
   title={Stochastic integration in Hilbert spaces with respect to
   cylindrical martingale-valued measures},
   journal={ALEA Lat. Am. J. Probab. Math. Stat.},
   volume={18},
   year={2021},
   number={2},
   pages={1267--1295},
   review={\MR{4282189}},
   doi={10.30757/alea.v18-47},
}



\bib{CCFM:SPDE}{article}{
   author={Cambronero, S.},
   author={Campos, D.},
   author={Fonseca-Mora, C. A.},
   author={Mena, D.},
   title={Cylindrical martingale-valued measures, stochastic integration and
   SPDEs},
   journal={Stoch. Partial Differ. Equ. Anal. Comput.},
   volume={13},
   year={2025},
   number={2},
   pages={887--955},
   issn={2194-0401},
   review={\MR{4908980}},
   doi={10.1007/s40072-024-00345-w},
}







\bib{CCFM:Ito}{article}{
   author={Cambronero, S.},
   author={Campos, D.},
   author={Fonseca-Mora, C. A.},
   author={Mena, D.},
   title={It\^{o}'s Formula for It\^{o} processes defined with respect to a cylindrical-martingale valued measure},
   journal = {Stoch. Anal. Appl.},
   year = {in press},
   doi={10.1080/07362994.2026.2701258},
  %eprint={2407.16086},
}


\bib{CohenElliott:2015}{book}{
   author={Cohen, S. N.},
   author={Elliott, R. J.},
   title={Stochastic calculus and applications},
   series={Probability and its Applications},
   edition={2},
   publisher={Springer, Cham},
   year={2015},
   pages={xxiii+666},
   isbn={978-1-4939-2866-8},
   isbn={978-1-4939-2867-5},
   review={\MR{3443368}},
   doi={10.1007/978-1-4939-2867-5},
}

\bib{DaPratoZabczyk}{book}{
   author={Da Prato, G.},
   author={Zabczyk, J.},
   title={Stochastic equations in infinite dimensions},
   series={Encyclopedia of Mathematics and its Applications},
   volume={152},
   edition={2},
   publisher={Cambridge University Press, Cambridge},
   year={2014},
   pages={xviii+493},
   isbn={978-1-107-05584-1},
   review={\MR{3236753}},
   doi={10.1017/CBO9781107295513},
}



\bib{GawaMand:2010}{book}{
   author={Gawarecki, L.},
   author={Mandrekar, V.},
   title={Stochastic differential equations in infinite dimensions with
   applications to stochastic partial differential equations},
   series={Probability and its Applications (New York)},
   publisher={Springer, Heidelberg},
   year={2011},
   pages={xvi+291},
   isbn={978-3-642-16193-3},
   review={\MR{2560625}},
   doi={10.1007/978-3-642-16194-0},
}





\bib{KarouiMeleard:1990}{article}{
   author={El Karoui, N.},
   author={M\'el\'eard, S.},
   title={Martingale measures and stochastic calculus},
   journal={Probab. Theory Related Fields},
   volume={84},
   year={1990},
   number={1},
   pages={83--101},
   issn={0178-8051},
   review={\MR{1027822}},
   doi={10.1007/BF01288560},
}




% \bib{Lu_Zhang}{book}{
%    author={L\"u, Qi},
%    author={Zhang, Xu},
%    title={Mathematical control theory for stochastic partial differential
%    equations},
%    series={Probability Theory and Stochastic Modelling},
%    volume={101},
%    publisher={Springer, Cham},
%    year={[2021] \copyright 2021},
%    pages={xiii+592},
%    isbn={978-3-030-82330-6},
%    isbn={978-3-030-82331-3},
%    review={\MR{4363403}},
%    doi={10.1007/978-3-030-82331-3},
% }


\bib{LeGallBrownian}{book}{
   author={Le Gall, J.-F.},
   title={Brownian motion, martingales, and stochastic calculus},
   series={Graduate Texts in Mathematics},
   volume={274},
   edition={Translated from the 2013 French edition},
   publisher={Springer, [Cham]},
   year={2016},
   pages={xiii+273},
   isbn={978-3-319-31088-6},
   isbn={978-3-319-31089-3},
   review={\MR{3497465}},
   doi={10.1007/978-3-319-31089-3},
}


\bib{MetivierPellaumail}{book}{
   author={M\'etivier, M.},
   author={Pellaumail, J.},
   title={Stochastic integration},
   series={Probability and Mathematical Statistics},
   publisher={Academic Press [Harcourt Brace Jovanovich, Publishers], New
   York-London-Toronto},
   year={1980},
   pages={xii+196},
   isbn={0-12-491450-0},
   review={\MR{0578177}},
}




\bib{Metivier}{book}{
   author={M\'etivier, M.},
   title={Semimartingales},
   series={De Gruyter Studies in Mathematics},
   volume={2},
   note={A course on stochastic processes},
   publisher={Walter de Gruyter \& Co., Berlin-New York},
   year={1982},
   pages={xi+287},
   isbn={3-11-008674-3},
   review={\MR{0688144}},
}




\bib{Ondrejat:2005}{article}{
   author={Ondrej\'at, M.},
   title={Brownian representations of cylindrical local martingales,
   martingale problem and strong Markov property of weak solutions of SPDEs
   in Banach spaces},
   journal={Czechoslovak Math. J.},
   volume={55(130)},
   year={2005},
   number={4},
   pages={1003--1039},
   issn={0011-4642},
   review={\MR{2184381}},
   doi={10.1007/s10587-005-0084-z},
}





\bib{PeszatZabczykSPDE}{book}{
   author={Peszat, S.},
   author={Zabczyk, J.},
   title={Stochastic partial differential equations with L\'evy noise},
   series={Encyclopedia of Mathematics and its Applications},
   volume={113},
   note={An evolution equation approach},
   publisher={Cambridge University Press, Cambridge},
   year={2007},
   pages={xii+419},
   isbn={978-0-521-87989-7},
   review={\MR{2356959}},
   doi={10.1017/CBO9780511721373},
}





\bib{VakhaniaTarieladzeChobanyan}{book}{
   author={Vakhania, N. N.},
   author={Tarieladze, V. I.},
   author={Chobanyan, S. A.},
   title={Probability distributions on Banach spaces},
   series={Mathematics and its Applications (Soviet Series)},
   volume={14},
   note={Translated from the Russian and with a preface by Wojbor A.
   Woyczynski},
   publisher={D. Reidel Publishing Co., Dordrecht},
   year={1987},
   pages={xxvi+482},
   isbn={90-277-2496-2},
   review={\MR{1435288}},
   doi={10.1007/978-94-009-3873-1},
}




\bib{VeraarYaroslavtsev:2016}{article}{
   author={Veraar, M.},
   author={Yaroslavtsev, I.},
   title={Cylindrical continuous martingales and stochastic integration in
   infinite dimensions},
   journal={Electron. J. Probab.},
   volume={21},
   year={2016},
   pages={Paper No. 59, 53},
   review={\MR{3563887}},
   doi={10.1214/16-EJP7},
}



\bib{Walsh:1986}{article}{
   author={Walsh, J. B.},
   title={An introduction to stochastic partial differential equations},
   conference={
      title={\'Ecole d'\'et\'e{} de probabilit\'es de Saint-Flour,
      XIV---1984},
   },
   book={
      series={Lecture Notes in Math.},
      volume={1180},
      publisher={Springer, Berlin},
   },
   isbn={3-540-16441-3},
   year={1986},
   pages={265--439},
   review={\MR{0876085}},
   doi={10.1007/BFb0074920},
}
        
    \end{biblist}
\end{bibdiv}
\end{document}